\documentclass[11pt]{article}

\usepackage{fullpage}
\usepackage{amsmath}
\usepackage{amssymb}
\usepackage{amsthm}
\usepackage{fancyhdr}
\usepackage{algorithm}
\usepackage{algorithmic}
\usepackage[usenames]{color}
\usepackage{multirow}
\usepackage{graphicx}
\usepackage{caption}
\usepackage{enumitem}
\usepackage{url}
\allowdisplaybreaks

\newtheorem{theorem}{Theorem}[section]
\newtheorem{lemma}[theorem]{Lemma}

\newtheorem{definition}[theorem]{Definition}

\newtheorem{remark}{Remark}
\def\la{\left\langle}
\def\ra{\right\rangle}
\def\ln{\left\|}
\def\rn{\right\|}
\def\lb{\left(}
\def\rb{\right)}

\def\lcb{\left\{}
\def\rcb{\right\}}
\def\lab{\left|}
\def\rab{\right|}
\def\subto{\mbox{ subject to }}
\def\A{\mathcal{A}}
\def\At{\mathcal{A}^*}
\def\R{\mathbb{R}}
\def\S{\mathcal{S}}
\def\P{\mathcal{P}}
\def\G{\mathcal{G}}
\def\E{\mathcal{E}}
\def\N{\mathcal{N}}
\def\H{\mathcal{H}}
\def\rank{\mbox{rank}}
\DeclareMathOperator*{\argmin}{arg\,min}

\newcommand\numberthis{\addtocounter{equation}{1}\tag{\theequation}}

\newcommand{\sigmin}[1]{\sigma_{\min}\lb #1\rb}
\newcommand{\sigmax}[1]{\sigma_{\max}\lb #1\rb}
\newcommand{\sigminsq}[1]{\sigma^2_{\min}\lb #1\rb}
\newcommand{\sigmaxsq}[1]{\sigma^2_{\max}\lb #1\rb}
\title{Guarantees of Riemannian Optimization for Low Rank Matrix Recovery}
\author{Ke Wei\thanks{Department of Mathematics, University of California at Davis, kewei@math.ucdavis.edu}\quad Jian-Feng Cai\thanks{Department of Mathematics, Hong Kong University of Science and Technology, \{jfcai,masyleung\}@ust.hk}\quad Tony F. Chan\thanks{Office of the President, Hong Kong University of Science and Technology, tonyfchan@ust.hk}\quad Shingyu Leung\footnotemark[2]}

\begin{document}
\maketitle
\begin{abstract}
We establish theoretical recovery guarantees of a family of Riemannian optimization algorithms 
for low rank matrix recovery, which  is about recovering an $m\times n$ rank $r$ matrix
from $p < mn$ number of linear measurements. The algorithms are first interpreted as  iterative hard thresholding algorithms with subspace projections. 
Based on this connection, we show that provided the restricted isometry constant  $R_{3r}$  of the sensing operator is less than $C_\kappa /\sqrt{r}$,  
the Riemannian gradient descent algorithm and a restarted variant of the Riemannian conjugate gradient algorithm are guaranteed to converge linearly
to the underlying rank $r$ matrix if they are initialized by one step hard thresholding. Empirical evaluation shows that the algorithms are able to recover a low 
rank matrix from nearly the minimum number of measurements necessary.
\end{abstract}

{\bf Keywords.} Matrix recovery, low rank matrix manifold, Riemannian optimization, gradient descent and conjugate gradient descent methods, restricted isometry constant

{\bf Mathematics Subject Classification.}  15A29, 41A29, 65F10,  68Q25, 15A83,  53B21, 90C26,  65K05
\section{Introduction}\label{sec:intro}
Many applications of interest require  acquisition of very high 
dimensional data which can be prohibitively
expensive if no  simple structures 
of the data are known. In contrast, data with  an inherent low dimensional 
structure can be acquired more efficiently by exploring the simplicity of 
the underlying structure. For instance, in compressed sensing \cite{donoho2006cs,crt2006robust,EJC1}, a high dimensional
 vector of length $n$ with only a few nonzero entries can be encoded by $p<n$ linear 
 measurements, where $p$ is essentially determined by the number of nonzero entries.  Moreover, 
 the  vector can be reconstructed  from 
 the  measurements by computationally efficient algorithms. Another natural representation of data is matrix which 
 can be imposed other different simple structures in addition to few nonzero entries.
 A particularly interesting  notion of matrix simplicity is low rank.
 
 {Low rank matrices can be used to model datasets from a wide range of applications, such as model reduction  \cite{LiVa2009interior}, pattern recognition \cite{Eld2007pattern}, and machine learning \cite{AmFiSrUl2007class,ArEvPo2007mult}. 
In this paper, we are interested in the problem of recovering a low rank matrix from a set of linear measurements.
Let $X\in\R^{m\times n}$ and assume $\rank(X)=r<\min(m,n)$. 
Let $\A\lb\cdot\rb:\R^{m\times n}\rightarrow\R^p$ be a linear map from $m\times n$ matrices to $p$ dimensional vectors of the form 
\begin{equation} 
\A\lb Z\rb_\ell = \la A_\ell,Z\ra,\quad \ell=1,\cdots,p.
\end{equation}
We take $p < mn$ linear measurements of $X$ via $y=\A(X)$. To recover the low rank matrix $X$ from  the measurement vector $y$, it is natural 
to seek the lowest rank matrix consistent with the  measurements by solving a rank minimization problem
\begin{equation}\label{eq:rank_min}
\min\rank(Z) \subto \A(Z) = y.
\end{equation}}There are two typical sensing operators. One is {\em dense sensing}, where each $A_\ell$ is a dense matrix, for example, a Gaussian 
random matrix. The other one is {\em entry sensing} with the sensing matrices $A_\ell$ only having one nonzero entry equal to one, which corresponds to measuring 
the entries of a matrix  directly. 
When a subset of the matrix entries are directly measured, seeking  a low rank matrix consistent with the known entries is typically referred to as {\em matrix completion} \cite{candesrecht2009mc}. Recovering a low rank matrix when each sensing matrix is dense is usually referred to as {\em low rank matrix recovery} \cite{candesplan2009oracle}. This paper investigates recovery guarantees of the Riemannian optimization algorithms  for low rank matrix recovery. 

Notice that \eqref{eq:rank_min} is a non-convex optimization problem and computationally intractable. One of the well studied approaches is to replace the 
rank objective in \eqref{eq:rank_min} with its nearest convex relaxation, the nuclear norm of matrices which is the sum of the singular values, and then solve the following nuclear norm minimization problem
\begin{equation}\label{eq:nnm_min}
\min\ln Z\rn_* \subto \A(Z) = y.
\end{equation}
The equivalence of solutions between
\eqref{eq:rank_min} and \eqref{eq:nnm_min} for low rank matrix recovery can be established in terms of the restricted isometry constant of the sensing operator which was first introduced in \cite{candestao2005decode} for compressed sensing
and subsequently extended to low rank matrix recovery in \cite{rechtfazelparrilo2010nnm}.
\begin{definition}[Restricted Isometry Constant (RIC) \cite{rechtfazelparrilo2010nnm}]\label{def:RIC}
Let $\A(\cdot)$ be a linear operator from $m\times n$ matrices to vectors of length $p$. For any integer $0<r\leq\min(m,n)$, the restricted isometry constant, 
$R_r$, is defined as the smallest number such that 
\begin{equation}
\lb1-R_r\rb\ln Z\rn_F^2\leq\ln \A\lb Z\rb\rn_2^2\leq\lb 1+R_r\rb\ln Z\rn_F^2
\end{equation}
holds for all the matrices $Z$ of rank at most $r$. 
\end{definition}

It has been proven  that if the RIC of $\A\lb \cdot\rb$ with rank $5r$ is less than a small constant, nuclear norm minization \eqref{eq:nnm_min} is guaranteed 
to recover any measured rank $r$ matrix \cite{rechtfazelparrilo2010nnm}. Furthermore, this condition can be satisfied with overwhelmingly high probability for a large family of random measurement matrices, for example  the normalized Gaussian
and Bernoulli  matrices, provided $p\geq C\cdot(m+n-r)r\log^\alpha\lb\max\lb m,n\rb\rb$ for some numerical constants $C>0,~\alpha\geq 0$ \cite{rechtfazelparrilo2010nnm, candesplan2009oracle}. However, for the entry sensing operator,
it cannot have a small RIC for any $r>0$, and more quantitative and probabilistic sampling complexity has been established for \eqref{eq:nnm_min} based on the notion of incoherence \cite{candesrecht2009mc, candestao2009mc, recht2011simple, gross2011recoverlowrank}.

Nuclear norm minimization is {amenable} to detailed analysis \cite{rechtfazelparrilo2010nnm, candesrecht2009mc, candestao2009mc, recht2011simple, gross2011recoverlowrank, brwxbh, oymakhassibi2010nullspace,rauhut_stable_low_rank, LivingEdge}. However, finding the solution to \eqref{eq:nnm_min} by the interior-point methods needs to solve systems of linear equations to 
compute the Newton direction in each iteration, which limits its applicability for large $m$ and $n$. First-order methods for solving \eqref{eq:nnm_min} usually invoke the singular value thresholding \cite{ccs2010svt}. Alternative to  convex relaxation, there have been many algorithms 
which are designed to target \eqref{eq:rank_min} directly, including iterative hard thresholding \cite{CGIHT, jmd2010svp,cevher2012matrixrecipes,tw2012nihtmc}, alternating minimization \cite{haldahernando2009pf, wyz2012lmafit, tannerwei_asd} , and Riemannian optimization \cite{bart2012riemannian,ngosaad2012scgrassmc,mishra_geometry,mmbs2012fixrank,R3MC,cambier}. {\em  In this paper, we study a family of Riemannian optimization algorithms  on the embedded manifold of rank $r$ matrices. We first establish their connections with iterative hard thresholding algorithms. Then we prove local convergence of the Riemannian gradient descent algorithm in terms of the
restricted isometry constant of the sensing operator. As a result, the algorithm is guaranteed to converge linearly to the measured low rank matrix if it is initialized by one step hard thresholding. For the Riemannian conjugate gradient descent algorithm, we introduce a restarted variant for which a similar recovery guarantee can be established while the computational 
efficiency is maintained.}

The rest of this paper is organized as follows. In Sec.~\ref{sec:iht}, we briefly review  iterative hard thresholding algorithms for low rank matrix recovery and show their connections with the 
Riemannian optimization algorithms on the embedded manifold of low rank matrices.  Then we present the main results of this paper. In Sec.~\ref{sec:numerics}, empirical results are presented, which demonstrate the efficiency of the 
Riemannian optimization algorithms for matrix recovery. Section~\ref{sec:analysis} presents the proofs of the main results and Sec.~\ref{sec:discuss} concludes the paper 
with future research directions.

\section{Algorithms and Main Results}\label{sec:iht}
\subsection{Iterative Hard Thresholding and Riemannian Optimization}
Iterative hard thresholding is a family of simple yet efficient algorithms for compressed sensing \cite{bludav2009iht,blumensathdavies2010niht,CGIHT,CGIHTnoise} and low rank matrix recovery \cite{CGIHT, jmd2010svp,cevher2012matrixrecipes,tw2012nihtmc}.
The simplest iterative hard thresholding algorithm for matrix recovery is the normalized iterative hard thresholding (NIHT \cite{tw2012nihtmc}, also known as SVP \cite{jmd2010svp} or IHT \cite{goldfarbma2011fpca} when the stepsize is fixed), see Alg.~\ref{alg:niht}. NIHT applies the
projected gradient descent method to a  reformulation of \eqref{eq:rank_min}
\begin{equation}\label{eq:reform}
\min_{Z\in\R^{m\times n}}\frac{1}{2}\ln y-\A\lb Z\rb\rn_2^2~\subto~\rank(Z)=r.
\end{equation}
\begin{algorithm}[htp]
\caption{Normalized Iterative Hard Thresholding (NIHT \cite{tw2012nihtmc}) }\label{alg:niht}
\begin{algorithmic}
\STATE\textbf{Initilization}: $X_0$ and its top $r$ left singular vector space $U_0$
\FOR{$l=0,1,\cdots$}
\STATE 1. $G_l=\At\lb y-\A\lb X_l\rb\rb$
\STATE 2. $\alpha_l=\frac{\ln\P_{U_l}\lb G_l\rb\rn_F^2}{\ln \A\P_{U_l}\lb G_l\rb\rn_2^2}$
\STATE 3. $W_l = X_l+\alpha_l G_l$
\STATE 4. $X_{l+1}=\H_r(W_l)$ 
\ENDFOR
\end{algorithmic}
\end{algorithm}
In each iteration of NIHT, the current estimate $X_l$ is updated along the gradient descent direction $G_l$ with {the locally steepest descent 
stepsize $\alpha_l$ defined as 
\begin{align*}
\alpha_l:=\argmin_\alpha\frac{1}{2}\ln y-\A\lb X_l+\alpha \P_{U_l}\lb G_l\rb\rb\rn_F^2,\numberthis\label{eq:niht_alpha}
\end{align*}
 where $\P_{U_l}=U_lU_l^*$ denotes the 
projection to the left singular vector subspace\footnote{{The left singular vector subspace of $X_l$ in Algs.~\ref{alg:niht} and Algs.~\ref{alg:cgiht} can be replaced by its right singular vector subspace, see \cite{tw2012nihtmc}.}} of $X_l$}. Then the new estimate $X_{l+1}$ is obtained by thresholding $W_l$ to the set of rank $r$ matrices. In Alg.~\ref{alg:niht}, $\H_r(\cdot)$ denotes 
the hard thresholding operator which first computes the singular value decomposition  of a matrix and then  sets all but the $r$ largest singular values to zero
\begin{equation}
\H_r(Z):=U\Sigma_r V^*\quad\mbox{where}\quad \Sigma_r(i,i):=\begin{cases}
\Sigma(i,i) & i\leq r\\
0 & i > r.
\end{cases}
\end{equation}
 {When there are singular values of $Z$ with multiplicity more than one, $\H_r(Z)$ can use any one of the repeated singular values and the corresponding singular vectors. In this case, it still returns the best (though not unique) rank $r$ approximation of $Z$ in the Frobenius norm.}

NIHT has been proven to be able to recover a rank $r$ matrix if the RIC of $\A(\cdot)$ satisfies $R_{3r}\leq 1/5$ \cite{tw2012nihtmc}. Despite 
the optimal recovery guarantee of NIHT, it suffers from the slow asymptotic convergence rate of the gradient descent method.
Other sophisticated variants have been designed to overcome the slow asymptotic convergence rate of NIHT. For example in 
SVP-Newton \cite{jmd2010svp}, a least square subproblem restricted onto the current iterate subspace is solved in each iteration. In \cite{CGIHT}, 
a family of conjugate gradient iterative hard thresholding  (CGIHT) algorithms have been developed for low rank matrix recovery
which combines the fast asymptotic convergence rate of more sophisticated algorithms and the low per iteration complexity
of NIHT, see Alg.~\ref{alg:cgiht} for the non restarted CGIHT.
\begin{algorithm}[htp]
\caption{Conjugate Gradient Iterative Hard Thresholding (CGIHT \cite{CGIHT})}\label{alg:cgiht}
\begin{algorithmic}
\STATE\textbf{Initilization}: $X_0$ and its top $r$ left singular vector space $U_0$, $P_{-1}=0$
\FOR{$l=0,1,\cdots$}
\STATE 1. $G_l=\At\lb y-\A\lb X_l\rb\rb$
\STATE 2. $\beta_l = -\frac{\la\A\P_{U_l}\lb G_l\rb,\A\P_{U_l}\lb P_{l-1}\rb\ra}{\ln\A\P_{U_l}\lb P_{l-1}\rb\rn_2^2}$
\STATE 3. $P_l=G_l+\beta_l P_{l-1}$
\STATE 4. $\alpha_l=\frac{\la \P_{U_l}\lb G_l\rb,\P_{U_l}\lb P_l\rb\ra}{\ln\A\P_{U_l}\lb P_l\rb\rn_2^2}$ 
\STATE 5. $W_l=X_l+\alpha_l P_l$
\STATE 6. $X_{l+1}=\H_r(W_l)$
\ENDFOR
\end{algorithmic}
\end{algorithm}

In each iteration of CGIHT, the current estimate $X_l$ is updated along the search direction $P_l$ with the locally steepest descent stepsize $\alpha_l$ defined in a similar way to \eqref{eq:niht_alpha}.
The current search direction $P_l$ is a linear combination of the gradient descent direction and the previous search direction. The selection of the orthogonalization weight 
$\beta_l$ in Alg.~\ref{alg:cgiht} ensures that $P_l$ is conjugate orthogonal to the $P_{l-1}$ when restricted to the current subspace determined by $U_l$. It has been proven that a projected
variant of Alg.~\ref{alg:cgiht} has the nearly optimal recovery guarantee  based on the RIC of the sensing operator \cite{CGIHT}. 

In NIHT and CGIHT, the current estimate is updated along a line search direction which departs from the manifold of rank $r$ matrices. The singular value 
decomposition is required in each iteration to project the estimate back onto the rank $r$ matrix manifold. 
The SVD on a full  $m\times n$ matrix is typically needed  as the search direction is a global gradient descent or conjugate gradient descent direction which does not belong to any particular low dimensional subspace. 
The computational complexity of the SVD on an $m\times n$ matrix is $O(n^3)$ when 
$m$ is proportional to $n$ which is computationally expensive. However, if the estimate is 
updated along a  search direction in a  low dimensional subspace, the intermediate matrix $W_l$
may also be a low rank matrix. So it is possible to work on a  matrix of size much smaller than $m$ and $n$
when truncating $W_l$ to its nearest rank $r$ approximation. The generalized NIHT and CGIHT for low rank matrix 
recovery \cite{KeWeiThesis} are presented in Algs.~\ref{alg:gniht} and \ref{alg:gcgiht} respectively, which explore 
the idea of projecting the search direction onto a low dimensional subspace associated with 
the current estimate. 
\begin{algorithm}[htp]
\caption{Riemannian Gradient Descent (RGrad)}\label{alg:gniht}
\begin{algorithmic}
\STATE\textbf{Initilization}: $X_0$
\FOR{$l=0,1,\cdots$}
\STATE 1. $G_l=\At\lb y-\A\lb X_l\rb\rb$
\STATE 2. $\alpha_l=\frac{\ln\P_{\S_l}\lb G_l\rb\rn_F^2}{\ln \A\P_{\S_l}\lb G_l\rb\rn_2^2}$
\STATE 3. $W_l = X_l+\alpha_l\P_{\S_l} \lb G_l\rb$
\STATE 4. $X_{l+1}=\H_r(W_l)$ 
\ENDFOR
\end{algorithmic}
\end{algorithm}

\begin{algorithm}[htp]
\caption{Riemannian Conjugate Gradient Descent (RCG)}\label{alg:gcgiht}
\begin{algorithmic}
\STATE\textbf{Initilization}: $X_0$, $\beta_0=0$ and $P_{-1}=0$ 
\FOR{$l=0,1,\cdots$}
\STATE 1. $G_l=\At\lb y-\A\lb X_l\rb\rb$
\STATE 2. $\beta_l = -\frac{\la\A\P_{\S_l}\lb G_l\rb,\A\P_{\S_l}\lb P_{l-1}\rb\ra}{\ln\A\P_{\S_l}\lb P_{l-1}\rb\rn_2^2}$
\STATE 3. $P_l=\P_{\S_l}\lb G_l\rb+\beta_l \P_{\S_l}\lb P_{l-1}\rb$
\STATE 4. $\alpha_l=\frac{\la \P_{\S_l}\lb G_l\rb,\P_{\S_l}\lb P_l\rb\ra}{\ln\A\P_{\S_l}\lb P_l\rb\rn_2^2}$ 
\STATE 5. $W_l=X_l+\alpha_l P_l$
\STATE 6. $X_{l+1}=\H_r(W_l)$
\ENDFOR
\end{algorithmic}
\end{algorithm}

Compared with NIHT in Alg.~\ref{alg:niht}, the major difference in Alg.~\ref{alg:gniht} is at step $3$, where 
the current estimate is updated along a  projected gradient descent direction rather than the gradient descent direction.
The search stepsize in Alg.~\ref{alg:gniht} is selected to be the steepest descent stepsize along the projected gradient descent direction.
In Alg.~\ref{alg:gcgiht}, the search direction is selected to be an appropriate 
linear combination of the projected gradient descent direction and the previous search direction projected 
onto the current  iterate subspace. 
As in Alg.~\ref{alg:cgiht}, the selection of $\beta_l$ in Alg.~\ref{alg:gcgiht} guarantees that the new search direction $P_l$
is conjugate orthogonal to the previous search direction $P_{l-1}$ when projected onto the subspace $\S_l$. Motivated 
by non-linear conjugate gradient method in optimization, there are other choices for $\beta_l$ \cite{AbMaSe2008manifold}, including
\begin{align*}
&\mbox{Fletcher-Reeves}\qquad \beta_l^{\mbox{FR}}=\frac{\ln\P_{\S_l}\lb G_l\rb\rn_F^2}{\ln\P_{\S_{l-1}}\lb G_{l-1}\rb\rn_F^2},\\
&\mbox{Polak-Ribi\`ere}\qquad \beta_l^{\mbox{PR}}=\frac{\la \P_{\S_l}\lb G_l\rb, \P_{\S_l}\lb G_l\rb-\P_{\S_l}\lb P_{\S_{l-1}}\lb G_{l-1}\rb\rb\ra}{\ln \P_{\S_{l-1}}\lb G_{l-1}\rb\rn_F^2},\numberthis\label{eq:beta_all}\\
&\mbox{Polak-Ribi\`ere+}\qquad\beta_l^{\mbox{PR+}}=\max\{\beta_l^{\mbox{PR}},0\}.
\end{align*}
\subsection{Selection of the Subspace $\S_l$}\label{subsec:subspace_selection}
Let $X_l$ be the  current rank $r$ estimate in Algs.~\ref{alg:gniht}~or~\ref{alg:gcgiht} and $X_l=U_l\Sigma_l V_l^*$ be the reduced
singular value decomposition with $U_l\in\R^{m\times r},~\Sigma\in\R^{r\times r}$ and $V_l\in\R^{n\times r}$. If $\S_l$ is selected to be the column space of $X_l$ 
\begin{equation}
\S_l=\lcb Z\in\R^{m\times n}:~Z = U_l R^*\mbox{ with }R\in\R^{n\times r}\rcb,
\end{equation}
then the intermediate matrix $W_l\in \S_l$ is already a rank $r$ matrix and $\S_l$ remains unchanged during all the iterations. So no hard thresholding is needed
to project $W_l$ to its nearest rank $r$ approximation. However, if $X\not\in \S_0$, the algorithm will never converge to $X$, but 
to a locally optimal solution.

Similar conclusion can be drawn when $\S_l$ is selected to be the row space of $X_l$
\begin{equation}
\S_l=\lcb Z\in\R^{m\times n}:~Z = L V_l^*\mbox{ with }L\in\R^{m\times r}\rcb.
\end{equation}
So it is desirable to use a  larger $\S_l$ in each iteration so that the subspace can be updated and can capture more and more information
of the underlying low rank matrix. A potential choice  is the direct sum of the column and row subspaces
\begin{equation}\label{eq:tangent_space}
\S_l=\lcb Z\in\R^{m\times n}:~Z = U_l R^*+L V_l^*\mbox{ with }L\in\R^{m\times r},~R\in\R^{n\times r}\rcb.
\end{equation}
The subspace $\S_l$ in \eqref{eq:tangent_space} turns out to be the tangent space of the smooth manifold of rank $r$ matrices at the current estimate $X_l$ \cite{bart2012riemannian}.  It is well known that all the $m\times n$ rank $r$ matrices form a smooth manifold of dimension $(m+n-r)r$ \cite{bart2012riemannian}, which coincides with the dimension of $\S_l$.   {With this subspace selection, Algs.~\ref{alg:gniht}~and~\ref{alg:gcgiht} are indeed the Riemannian gradient descent and conjugate gradient 
descent algorithms  on the embedded manifold of rank $r$ matrices under the metric of canonical matrix inner product.} The projection of the previous search direction onto the current tangent subspace, $\P_{\S_l}P_{l-1}$, corresponds to ``vector transport'' in  Riemannian optimization and the hard thresholding operator corresponds to a type of ``retraction''. The Riemannian conjugate 
gradient algorithm for matrix completion developed in \cite{bart2012riemannian} can be recovered from Alg.~\ref{alg:gcgiht} with the selection of $\beta_l$ being replaced by $\beta^{\mbox{PR+}}$ in \eqref{eq:beta_all}. For more details about Riemannian optimization, we refer the reader to \cite{AbMaSe2008manifold}.  {In particular, the differential geometry interpretations of the Riemannian optimization algorithms on  the embedded manifold of low rank matrices can be found in \cite{bart2012riemannian}.}

\subsection{SVD of $W_l$ with $O(r^3)$ Complexity}
 In the sequel, we  assume $\S_l$ is selected to be the tangent space specified in \eqref{eq:tangent_space}, unless stated otherwise. So Algs.~\ref{alg:gniht}~and~\ref{alg:gcgiht} are  the Riemannian gradient descent and conjugate gradient descent algorithms. First notice that the matrices in $\S_l$ 
are at most rank $2r$. In addition, for any matrix $Z\in\R^{m\times n}$, the projection of $Z$ onto $\S_l$ can be computed as 
\begin{equation}\label{eq:S_proj}
\P_{\S_l}\lb Z\rb = U_lU_l^*Z + ZV_lV_l^*-U_lU_l^*ZV_lV_l^*.
\end{equation}
In  Algs.~\ref{alg:gniht}~and~\ref{alg:gcgiht}, the SVD is still required when projecting $W_l$ onto the rank $r$  manifold since it is not a rank $r$ matrix. 
However, as $W_l$ is in the low dimensional subspace $\S_l$, the SVD of $W_l$ can be computed from the SVD of a smaller size matrix. To see this,
notice that the intermediate matrix  $W_l$ in Algs.~\ref{alg:gniht}~and~\ref{alg:gcgiht} has the form
\begin{equation*}
W_l = X_l + \P_{\S_l}\lb Z_l\rb,
\end{equation*}
where $Z_l=\alpha_lG_l$ in Alg.~\ref{alg:gniht} and $Z_l=\alpha_l\lb G_l+\beta_l P_{l-1}\rb$ in Alg.~\ref{alg:gcgiht}. So direct calculation gives 
\begin{align*}
W_l &=X_l + \P_{\S_l}\lb Z_l\rb\\
&=U_l\Sigma_l V_l^*+U_lU_l^*Z_l + Z_lV_lV_l^*-U_lU_l^*Z_lV_lV_l^*\\
&=U_l\Sigma_l V_l^*+U_lU_l^*Z_lV_lV_l^*
+U_lU_l^*Z_l-U_lU_l^*Z_lV_lV_l^*
+Z_lV_lV_l^*-U_lU_l^*Z_lV_lV_l^*\\
&=U_l\lb \Sigma_l+U_l^*Z_lV_l\rb V_l^*+U_lU_l^*Z_l\lb I-V_lV_l^*\rb+\lb I-U_lU_l^*\rb Z_lV_lV_l^*\\
&:=U_l\lb \Sigma_l+U_l^*Z_lV_l\rb V_l^*+U_lY_1^*+Y_2V_l^*.
\end{align*}
Let $Y_1=Q_1R_1$ and $Y_2=Q_2R_2$ be the QR factorizations of $Y_1$ and $Y_2$ respectively. Then we have $U_l^*Q_2=0$, $V_l^*Q_1=0$ and $W_l$ can be rewritten as 
\begin{align*}
W_l &=U_l\lb \Sigma_l+U_l^*Z_lV_l\rb V_l^*+U_lR_1^*Q_1^*+Q_2R_2V_l^*\\
&=\begin{bmatrix}U_l & Q_2\end{bmatrix}
\begin{bmatrix}
\Sigma_l+U_l^*Z_lV_l & R_1^*\\
R_2 & 0
\end{bmatrix}
\begin{bmatrix}
V_l^*\\
Q_1^*
\end{bmatrix}\\
&:=\begin{bmatrix}U_l & Q_2\end{bmatrix}
M_l
\begin{bmatrix}
V_l^*\\
Q_1^*
\end{bmatrix},
\end{align*}
where $M_l$ is a $2r\times 2r$ matrix.  Since $\begin{bmatrix}U_l & Q_2\end{bmatrix}$ and $\begin{bmatrix}V_l & Q_1\end{bmatrix}$ are 
both orthogonal matrices, the SVD of $W_l$ can be obtained from the SVD of $M_l$, which can be computed using $O(r^3)$ floating point operations (flops) instead of $O(n^3)$ flops.
\subsection{Main Results}
We first present recovery guarantee of the Riemannian gradient descent algorithm (Alg.~\ref{alg:gniht}) in terms of the restricted isometry constant of the sensing operator.  
\begin{theorem}[{Recovery} guarantee of Riemannian gradient descent (Alg.~\ref{alg:gniht}) for low rank matrix recovery]\label{thm:gniht_dense} Let $\A(\cdot)$ be a linear map 
from $\R^{m\times n}$ to $\R^p$ with $p<mn$, and $y=\A(X)$ with $\rank(X)=r$. Define the following constant
{\begin{equation}\label{eq:gniht_dense_gamma_thm}
\gamma=\frac{4R_{2r}+2R_{3r}}{1-R_{2r}}+\frac{4R_{2r}}{\sigmin{X}}\ln X\rn_F.\end{equation}}
Then provided $\gamma<1$, the iterates of Alg.~\ref{alg:gniht} with initial point $X_0=\H_r\lb\At\lb y\rb\rb$ satisfy
\begin{equation}
\ln X_l-X\rn_F\leq \mu^l\ln X_0-X\rn_F,
\end{equation}
where $\mu=\gamma$.
In particular, $\gamma<1$ can be satisfied if
{\begin{equation}\label{eq:rgrad_ric_cond}
R_{3r}\leq \frac{\sigmin{X}}{\sigmax{X}}\cdot\frac{1}{12\sqrt{r}}.
\end{equation}}
\end{theorem}
For the Riemannian conjugate gradient descent method, we consider a restarted variant of Alg.~\ref{alg:gcgiht}. In 
the restarted Riemannian conjugate gradient descent algorithm,  $\beta_l$ is set to zero and restarting occurs whenever one of the following two conditions is violated  
\begin{equation}\label{eq:restart_cond_rm}
{\frac{\lab\la\P_{\S_l}\lb G_l\rb, \P_{\S_l}\lb P_{l-1}\rb\ra\rab}{\ln\P_{\S_l}\lb G_l\rb \rn_F\ln \P_{\S_l}\lb P_{l-1}\rb\rn_F}\leq \kappa_1} \quad\mbox{and}\quad
 \ln \P_{\S_l}\lb G_l\rb\rn_F\leq \kappa_2\ln \P_{\S_l}\lb P_{l-1}\rb\rn_F,
\end{equation}
where $0<\kappa_1 <1$ and $\kappa_2\geq 1$ are numerical constants; otherwise, $\beta_l$ is computed using the formula in Alg.~\ref{alg:gcgiht}.
We want to emphasize that the restarting conditions are introduced not only for the sake of proof, but also to improve the robustness 
of the non-linear conjugate gradient method \cite{wrightnocedal2006no, bart2012riemannian}. The first condition guarantees that the residual should be substantially orthogonal to the previous search direction 
when projected onto the current iterate subspace so that the new search direction can be sufficiently gradient related. In the classical conjugate gradient method 
for least square systems, the current residual is exactly orthogonal to the previous search direction. The second condition implies that the current residual cannot be too large 
when compared with the projection of the previous search direction which is in turn proportional to the projection of the previous residual. In
our implementations, we take $\kappa_1=0.1$ and $\kappa_2=1$.
\begin{theorem}[Recovery guarantee of restarted Riemannian conjugate gradient descent (Alg.~\ref{alg:gcgiht}) for low rank matrix recovery]\label{thm:gcgiht_dense} Let $\A(\cdot)$ be a linear map 
from $\R^{m\times n}$ to $\R^p$ with $p<mn$, and $y=\A(X)$ with $\rank(X)=r$. Define the following constants
\begin{align*}
&\varepsilon_\alpha = \frac{R_{2r}}{(1-R_{2r})-\kappa_1(1+R_{2r})},\quad\varepsilon_\beta = \frac{\kappa_2R_{2r}}{1-R_{2r}}+\frac{\kappa_1\kappa_2}{1-R_{2r}},\\
&{\tau_1 = 2(R_{2r}+R_{3r})(1+\varepsilon_\alpha)+2\varepsilon_\alpha+\frac{4R_{2r}}{\sigmin{X}}\ln X\rn_F+\varepsilon_\beta},\\&\tau_2=2\varepsilon_\beta \lb 1+\varepsilon_\alpha\rb\lb 1+R_{2r}\rb,\\
&\gamma=\tau_1+\tau_2.\numberthis\label{eq:gcgiht_dense_gamma_thm}
\end{align*}
Then provided $\gamma<1$, the iterates of the restarted conjugate gradient descent algorithm  (Alg.~\ref{alg:gcgiht} restarting subject to the conditions listed in \eqref{eq:restart_cond_rm}) with initial point $X_0=\H_r\lb\At\lb y\rb\rb$ satisfy
\begin{equation}
\ln X_l-X\rn_F\leq \mu^l\ln X_0-X\rn_F,
\end{equation}
where $\mu=\frac{1}{2}\lb\tau_1+\sqrt{\tau_1^2+4\tau_2}\rb<1$.
Moreover, when $\kappa_1=\kappa_2=0$, $\gamma$ in \eqref{eq:gcgiht_dense_gamma_thm} is equal to that in \eqref{eq:gniht_dense_gamma_thm}. On the other hand, we have 
\begin{equation*}
\lim_{R_{2r}, R_{3r}\rightarrow 0} \gamma = 3\kappa_1\kappa_2,
\end{equation*}
So whenever $\kappa_1\kappa_2<1/3$,  $\gamma$ can be less than $1$ if the restricted isometry constants $R_{2r}$ and $R_{3r}$ of the sensing operator are small. In particular, if $\kappa_1=0.1$ and $\kappa_2=1$, 
a sufficient condition for $\gamma<1$ is
{\begin{equation}\label{eq:rcg_ric_cond}
R_{3r}\leq \frac{\sigmin{X}}{\sigmax{X}}\cdot\frac{1}{25\sqrt{r}}.
\end{equation}}
\end{theorem}

\begin{remark}\label{remark1}{\normalfont
 {The RIC conditions in Thms.~\ref{thm:gniht_dense} and \ref{thm:gcgiht_dense} are more stringent than conditions of the form $R_{3r}<c$, where $c>0$ is a universal numerical constant,  but how much stringent is it? Let us consider a random measurement model which satisfies the concentration inequality (II.2) in \cite{candesplan2009oracle}.  Then the proof of Thm.~2.3 in \cite{candesplan2009oracle}
 reveals that \eqref{eq:rgrad_ric_cond} and \eqref{eq:rcg_ric_cond} will hold with high probability if the number of measurements is $O\lb \max(m,n)r\log\lb\frac{\sigmax{X}}{\sigmin{X}}\sqrt{r}\rb\rb$, 
  which implies the sampling complexity is nearly optimal up to  a logarithm factor. } }
\end{remark}
\begin{remark}\label{remark2}{\normalfont
 {As stated previously, the entry sensing operator in matrix completion cannot have a small RIC for any $r>0$. So Thms.~\ref{thm:gniht_dense} and \ref{thm:gcgiht_dense}  cannot be applied to justify the 
 success of the Riemannian gradient descent and conjugate gradient algorithms for matrix completion. Recently, Wei et al. \cite{CGIHT_entry} provided the first recovery guarantees of Algs.~\ref{alg:gniht} and \ref{alg:gcgiht} for matrix completion. In a nutshell,  if the number of known entries is $O(\max(m,n)r^2\log^2(\max(m,n)))$, Algs.~\ref{alg:gniht} and \ref{alg:gcgiht} with good initial guess are able to recover an incoherent low rank matrix with high probability. }}
\end{remark}
 \begin{remark}\label{remark3}{\normalfont
 If we first run $O(\log r)+O(\log(\sigma_{\max}(X)/\sigma_{\min}(X)))$ iterations of Alg.~\ref{alg:niht} until 
\begin{equation*}
\frac{2}{\sigmin{X}}\ln X_l-X\rn_F <1,
\end{equation*}
and then switch to Alg.~\ref{alg:gniht}, we have  
\begin{align*}
\frac{4R_{2r}+2R_{3r}}{1-R_{2r}}+\frac{2}{\sigmin{X}}\ln X_l-X\rn_F < 1
\end{align*}
 if $R_{3r}<c$ for a sufficiently small numerical constant $c>0$.
So it follows from \eqref{eq:recursive_gniht_dense_null} that the sufficient condition for  successful recovery of Alg.~\ref{alg:gniht} can be reduced to $R_{3r}<c$. Similar initialization scheme can be applied to Alg.~\ref{alg:gcgiht}. However, this is not advocated as it is difficult to determine the switching point in practice.
 }
 \end{remark}
  \begin{remark}\label{remark4}{\normalfont
 In Sec.~\ref{subsec:subspace_selection}, we have noted that if the subspace $\S_l$ is selected to be the column space or the row space of $X_l$, Algs.~\ref{alg:gniht} and \ref{alg:gcgiht} will not work. Here it is worth pointing out why the proofs for Thms.~\ref{thm:gniht_dense} and \ref{thm:gcgiht_dense} no longer hold if the tangent space is replaced by either the column space or the row space of the current iterate. The key to the proofs of Thms.~\ref{thm:gniht_dense} and \ref{thm:gcgiht_dense} is that $\ln \lb I-\P_{\S_l}\rb\lb X_l-X\rb\rn_F=o\lb\ln X_l-X\rn_F\rb$. However, if for example $\S_l$ is selected to be the column space of $X_l$, then 
\begin{eqnarray*}
\ln \lb I-\P_{\S_l}\rb\lb X_l-X\rb\rn_F &=&\ln \lb I-\P_{U_l}\rb X \rn_F\\
&=&\ln\lb UU^*-U_lU_l^*\rb X\rn_F\leq \ln UU^*-U_lU_l^*\rn_F\ln X\rn_2\\
&\leq&\frac{\sqrt{2}\ln X_l-X\rn_F\ln X\rn_2}{\sigmin{X}},
\end{eqnarray*}
where last inequality follows from Lem.~\ref{lem:for_proj_error}. This implies $\ln \lb I-\P_{\S_l}\rb\lb X_l-X\rb\rn_F$ is no longer a lower order of $\ln X_l-X\rn_F$ when $\S_l$
is the column space of $X_l$.}
\end{remark}
  \begin{remark}\label{remark5}{\normalfont
 {There has been a growing interest in investigating recovery guarantees of fast non-convex algorithms for both low rank matrix recovery  \cite{procrstes_flow,chenwain_fast,zhenglaff,saolre,bhokyrsan,zhaowangliu,whitesangward,JaNeSa2012ammc} and matrix completion \cite{CGIHT_entry,rhkamso,rhkamso_noisy,sunluo,kes_thesis, JaNeSa2012ammc,JaNe2015fast}. We compare the results in Thms.~\ref{thm:gniht_dense} and \ref{thm:gcgiht_dense} with those in \cite{JaNeSa2012ammc, procrstes_flow} where theoretical guarantees are established for other recovery algorithms using restricted isometry constant, and indirect comparisons can be made from them.  
It has been proven in \cite{JaNeSa2012ammc} that if $R_{2r}\leq \frac{\sigminsq{X}}{\sigmaxsq{X}}\cdot \frac{1}{100 r}$, alternating minimization initialized by one step hard thresholding is guaranteed to recover the underlying low rank matrix. This result is similar to recovery guarantees in Thms.~\ref{thm:gniht_dense} and \ref{thm:gcgiht_dense} if interpreted in terms of the sampling complexity (see Remark~\ref{remark1}). The gradient descent algorithm based on the product factorization of low rank matrices is shown to be able to converge linearly to the measured rank $r$ matrix if $R_{6r}<\frac{1}{10}$ and the algorithm is initialized by running (N)IHT for  a logarithm number of iterations \cite{ procrstes_flow}.  Remark~\ref{remark3} shows that this is also true for the Riemannian gradient descent and conjugate gradient descent algorithms.
For theoretical comparisons between the algorithms discussed in this paper and other algorithms in literature on matrix completion, we refer the reader to \cite{CGIHT_entry}.
}
}
\end{remark}

\section{Numerical Experiments}\label{sec:numerics}
In this section, we present empirical observations of the Riemannian gradient descent and conjugate gradient descent algorithms.  The numerical experiments 
are conducted on a Mac Pro laptop with 2.5GHz quad-core Intel Core i7 CPUs and 16 GB memory
and executed from Matlab 2014b. The tests presented in this section focus on square
matrices as is typical in the literature.
\subsection{Empirical Phase Transition}
A central question in matrix recovery is that given a triple $(m,n,r)$ how many of measurements are needed in order for an 
algorithm to be able to reliably recover a low rank matrix.  Though the theoretical results in Thms.~\ref{thm:gniht_dense} and \ref{thm:gcgiht_dense} can provide sufficient conditions for recovery, they are typically pessimistic when compared with the empirical observations. In practice, 
we evaluate the {recover} ability of an algorithm in  the phase transition framework, which compares the number of measurements, $p$, the size of an $m\times n$ matrix, 
$mn$, and the minimum number of measurements required to recover a rank $r$ matrix, $(m+n-r)r$, through the {\em undersampling} and {\em oversampling} ratios
\begin{equation}\label{eq:delta_rho}
\delta=\frac{p}{mn},\quad\rho=\frac{(m+n-r)r}{p}.
\end{equation}
The phase transition curve separates the $(\delta,\rho)\in[0,1]^2$ plane into two regions.
For problem instances with $(\delta,\rho)$ below the phase transition curve, the algorithm
is observed to be able to converge to the measured matrix. On the other hand, for problem instances
with $(\delta,\rho)$ above the phase transition curve, the algorithm is observed to return a solution that does not 
match the measured matrix.

Recall that in matrix recovery a linear operator consists of  a number of measurement matrices, each of which returns  
a measurement by taking inner product with the measured matrix. Though the model of entry sensing does not satisfy the 
RIC condition, the algorithms work equally well for it. So we conduct tests on the following two representative sensing 
operators: 
\begin{itemize}
\item $\G$: each entry of the sensing matrix is sampled from the standard Gaussian distribution $\N(0,1)$;
\item $\E$: a subset of entries of the measured matrix are sampled uniformly at random.
\end{itemize}
The test rank $r$ matrix $X\in\R^{m\times n}$ is formed as the product of two random rank $r$ matrices; that is $X=LR$, where 
$L\in\R^{m\times r}$ and $R\in\R^{r\times n}$ with $L$ and $R$ having their entries sampled from the standard 
Gaussian distribution $\N(0,1)$. In the tests, an algorithm is considered to have successfully recovered the rank $r$ matrix $X$
if it returns a matrix $X_l$ which satisfies 
\begin{equation*}
\frac{\ln X_l-X\rn_F}{\ln X\rn_F}\leq 10^{-2}.
\end{equation*}

We present the empirical phase transition of the Riemannian gradient descent algorithm and the Riemannian conjugate gradient
descent algorithms with and without restarting. In the restarted Riemannian conjugate gradient algorithm, we set { $\kappa_1=0.1$ and $\kappa_2=1$}. 
The tests are conducted with the undersampling ratio $\delta=p/mn$ taking 
$18$ equispaced values from $0.1$ to $0.95$. For Gaussian sensing, we conduct tests with $m=n=80$ while for entry sensing 
$m=n=800$. For each triple $(m,n,p)$, we start from a rank $r$ that is sufficiently small so that the algorithm can recover all the test 
matrices in {\em ten} random tests.\footnote{A larger number of random tests  have been conducted for a subset of problems. It was observed that
the entries in Tabs.~\ref{table:phase_gaussian} and \ref{table:phase_entry} didn't change significantly.} Then we increase the rank by $1$ until it reaches a value such that the algorithm fails to recover 
each of the {\em ten} test matrices. We refer to the largest rank that the algorithm succeeds in recovering  all the test matrices as $r_{\min}$, 
and the smallest rank that the algorithm fails all the tests as $r_{\max}$. The values of $r_{min}$, $r_{max}$ and the associated $\rho_{min}$,
$\rho_{max}$ computed through \eqref{eq:delta_rho} are listed in Tab.~\ref{table:phase_gaussian} for Gaussian sensing and in Tab.~\ref{table:phase_entry} for entry sensing. Figure~\ref{phase_mc} presents the average empirical phase transition curves of 
the tested algorithms on the $(\delta,\rho)$ plane, where we use $r=(r_{\min}+r_{\max})/2$ to compute the oversampling ratio $\rho$.
\begin{table}[htp]
\footnotesize
\centering
\caption{Phase transition table for Gaussian sensing with $m=n=80$. For each $(m,n,p)$ with $p=\delta\cdot mn$, the algorithm 
can recover all of the ten random test matrices when $r\leq r_{\min}$, but fails to recover each of the randomly drawn matrices when $r\geq r_{\max}$.}\label{table:phase_gaussian}
\vspace{0.2cm}
\begin{tabular}{|c|cccc|cccc|cccc|}
\hline
& \multicolumn{4}{|c|}{RGrad} & \multicolumn{4}{|c|}{RCG} & \multicolumn{4}{|c|}{RCG restarted}\\
\hline
$\delta$ & $r_{\min}$ & $r_{\max}$ & $\rho_{\min} $ & $\rho_{\max}$& $r_{\min}$ & $r_{\max}$ & $\rho_{\min} $ & $\rho_{\max}$& $r_{\min}$ & $r_{\max}$ & $\rho_{\min} $ & $\rho_{\max}$\\
\hline
0.1 & 3 & 4 & 0.74 & 0.97 & 3 & 4 & 0.74 & 0.97 & 3 & 4 & 0.74 & 0.97 \\
\hline
0.15 & 4 & 6 & 0.65 & 0.96 & 4 & 6 & 0.65 & 0.96 & 4 & 6 & 0.65 & 0.96 \\
\hline
0.2 & 6 & 8 & 0.72 & 0.95 & 6 & 8 & 0.72 & 0.95 & 6 & 8 & 0.72 & 0.95 \\
\hline
0.25 & 8 & 10 & 0.76 & 0.94 & 8 & 10 & 0.76 & 0.94 & 8 & 10 & 0.76 & 0.94 \\
\hline
0.3 & 11 & 12 & 0.85 & 0.93 & 11 & 13 & 0.85 & 1 & 11 & 13 & 0.85 & 1 \\
\hline
0.35 & 12 & 15 & 0.79 & 0.97 & 12 & 15 & 0.79 & 0.97 & 11 & 15 & 0.73 & 0.97 \\
\hline
0.4 & 14 & 17 & 0.8 & 0.95 & 14 & 17 & 0.8 & 0.95 & 14 & 17 & 0.8 & 0.95 \\
\hline
0.45 & 17 & 19 & 0.84 & 0.93 & 17 & 19 & 0.84 & 0.93 & 17 & 19 & 0.84 & 0.93 \\
\hline
0.5 & 20 & 22 & 0.88 & 0.95 & 20 & 22 & 0.88 & 0.95 & 20 & 22 & 0.88 & 0.95 \\
\hline
0.55 & 22 & 24 & 0.86 & 0.93 & 22 & 24 & 0.86 & 0.93 & 22 & 24 & 0.86 & 0.93 \\
\hline
0.6 & 25 & 27 & 0.88 & 0.94 & 26 & 28 & 0.91 & 0.96 & 26 & 28 & 0.91 & 0.96 \\
\hline
0.65 & 28 & 30 & 0.89 & 0.94 & 28 & 32 & 0.89 & 0.98 & 28 & 32 & 0.89 & 0.98 \\
\hline
0.7 & 31 & 33 & 0.89 & 0.94 & 31 & 35 & 0.89 & 0.98 & 31 & 35 & 0.89 & 0.98 \\
\hline
0.75 & 34 & 36 & 0.89 & 0.93 & 35 & 38 & 0.91 & 0.97 & 35 & 38 & 0.91 & 0.97 \\
\hline
0.8 & 38 & 40 & 0.91 & 0.94 & 40 & 42 & 0.94 & 0.97 & 40 & 42 & 0.94 & 0.97 \\
\hline
0.85 & 42 & 44 & 0.91 & 0.94 & 44 & 47 & 0.94 & 0.98 & 44 & 47 & 0.94 & 0.98 \\
\hline
0.9 & 47 & 50 & 0.92 & 0.95 & 50 & 53 & 0.95 & 0.98 & 50 & 53 & 0.95 & 0.98 \\
\hline
0.95 & 52 & 54 & 0.92 & 0.94 & 57 & 61 & 0.97 & 0.99 & 57 & 61 & 0.97 & 0.99 \\
\hline\end{tabular}\end{table}

Tables~\ref{table:phase_gaussian}, \ref{table:phase_entry} and Fig.~\ref{phase_mc} show that all the three tested algorithms are 
able to recover a rank $r$ matrix from $p=C\cdot (m+n-r)r$ number of measurements with $C$ being slightly larger than $1$. The ability 
of reconstructing a low rank matrix from nearly the minimum number of measurements has been previously reported in \cite{tw2012nihtmc, tannerwei_asd, CGIHT}
for other algorithms on low rank matrix recovery and matrix completion. 
To be more precise, Fig.~\ref{phase_mc} shows that the phase transition curves of RCG and RCG restarted are almost indistinguishable 
to each other, which shows the effectiveness of our restarting conditions. For Gaussian sensing, the phase transition curves of RCG and RCG 
restarted are slightly higher than that of RGrad when $\delta\geq 0.6$, while for entry sensing RGrad has a slightly higher phase transition curve when 
$\delta$ is small. Despite that, Tabs.~\ref{table:phase_gaussian} and \ref{table:phase_entry} show that their recovery performance only differs by 
one or two ranks. The erratic behavior of the phase transition curves for Gaussian sensing is due to the small value of $m=n=80$ and associated 
large changes in $\rho$ for a rank one change.

\begin{table}[htp]
\footnotesize
\centering
\caption{Phase transition table for entry sensing with $m=n=800$. For each $(m,n,p)$ with $p=\delta\cdot mn$, the algorithm can recover all of the ten random test matrices when $r\leq r_{\min}$, but failes to recover each of the randomly drawn matrices when $r\geq r_{\max}$.}\label{table:phase_entry}
\vspace{0.2cm}

\begin{tabular}{|c|cccc|cccc|cccc|}
\hline
& \multicolumn{4}{|c|}{RGrad} & \multicolumn{4}{|c|}{RCG} & \multicolumn{4}{|c|}{RCG restarted}\\
\hline
$\delta$ & $r_{\min}$ & $r_{\max}$ & $\rho_{\min} $ & $\rho_{\max}$& $r_{\min}$ & $r_{\max}$ & $\rho_{\min} $ & $\rho_{\max}$& $r_{\min}$ & $r_{\max}$ & $\rho_{\min} $ & $\rho_{\max}$\\
\hline
0.1 & 36 & 38 & 0.88 & 0.93 & 35 & 37 & 0.86 & 0.9 & 36 & 37 & 0.88 & 0.9 \\
\hline
0.15 & 55 & 59 & 0.89 & 0.95 & 55 & 57 & 0.89 & 0.92 & 55 & 57 & 0.89 & 0.92 \\
\hline
0.2 & 76 & 78 & 0.9 & 0.93 & 74 & 77 & 0.88 & 0.92 & 74 & 77 & 0.88 & 0.92 \\
\hline
0.25 & 97 & 99 & 0.91 & 0.93 & 96 & 98 & 0.9 & 0.92 & 96 & 98 & 0.9 & 0.92 \\
\hline
0.3 & 119 & 121 & 0.92 & 0.93 & 117 & 119 & 0.9 & 0.92 & 117 & 119 & 0.9 & 0.92 \\
\hline
0.35 & 142 & 143 & 0.92 & 0.93 & 140 & 142 & 0.91 & 0.92 & 140 & 142 & 0.91 & 0.92 \\
\hline
0.4 & 166 & 167 & 0.93 & 0.93 & 163 & 166 & 0.91 & 0.93 & 163 & 166 & 0.91 & 0.93 \\
\hline
0.45 & 190 & 192 & 0.93 & 0.94 & 188 & 191 & 0.92 & 0.93 & 188 & 191 & 0.92 & 0.93 \\
\hline
0.5 & 217 & 219 & 0.94 & 0.95 & 214 & 217 & 0.93 & 0.94 & 214 & 217 & 0.93 & 0.94 \\
\hline
0.55 & 244 & 248 & 0.94 & 0.95 & 242 & 246 & 0.93 & 0.95 & 242 & 245 & 0.93 & 0.94 \\
\hline
0.6 & 274 & 276 & 0.95 & 0.95 & 272 & 274 & 0.94 & 0.95 & 272 & 274 & 0.94 & 0.95 \\
\hline
0.65 & 306 & 308 & 0.95 & 0.96 & 302 & 306 & 0.94 & 0.95 & 304 & 306 & 0.95 & 0.95 \\
\hline
0.7 & 340 & 343 & 0.96 & 0.96 & 338 & 340 & 0.95 & 0.96 & 338 & 340 & 0.95 & 0.96 \\
\hline
0.75 & 378 & 380 & 0.96 & 0.97 & 374 & 378 & 0.96 & 0.96 & 374 & 378 & 0.96 & 0.96 \\
\hline
0.8 & 418 & 422 & 0.96 & 0.97 & 416 & 420 & 0.96 & 0.97 & 416 & 420 & 0.96 & 0.97 \\
\hline
0.85 & 466 & 470 & 0.97 & 0.98 & 464 & 468 & 0.97 & 0.97 & 464 & 468 & 0.97 & 0.97 \\
\hline
0.9 & 524 & 527 & 0.98 & 0.98 & 522 & 526 & 0.98 & 0.98 & 522 & 526 & 0.98 & 0.98 \\
\hline
0.95 & 600 & 604 & 0.99 & 0.99 & 600 & 604 & 0.99 & 0.99 & 600 & 604 & 0.99 & 0.99 \\
\hline\end{tabular}\end{table}

\begin{figure}
\centering

\end{figure}
\begin{figure}[htp]
\centering
\begin{tabular}{cc} 
\vspace{0.1in}
\includegraphics[scale=.35]{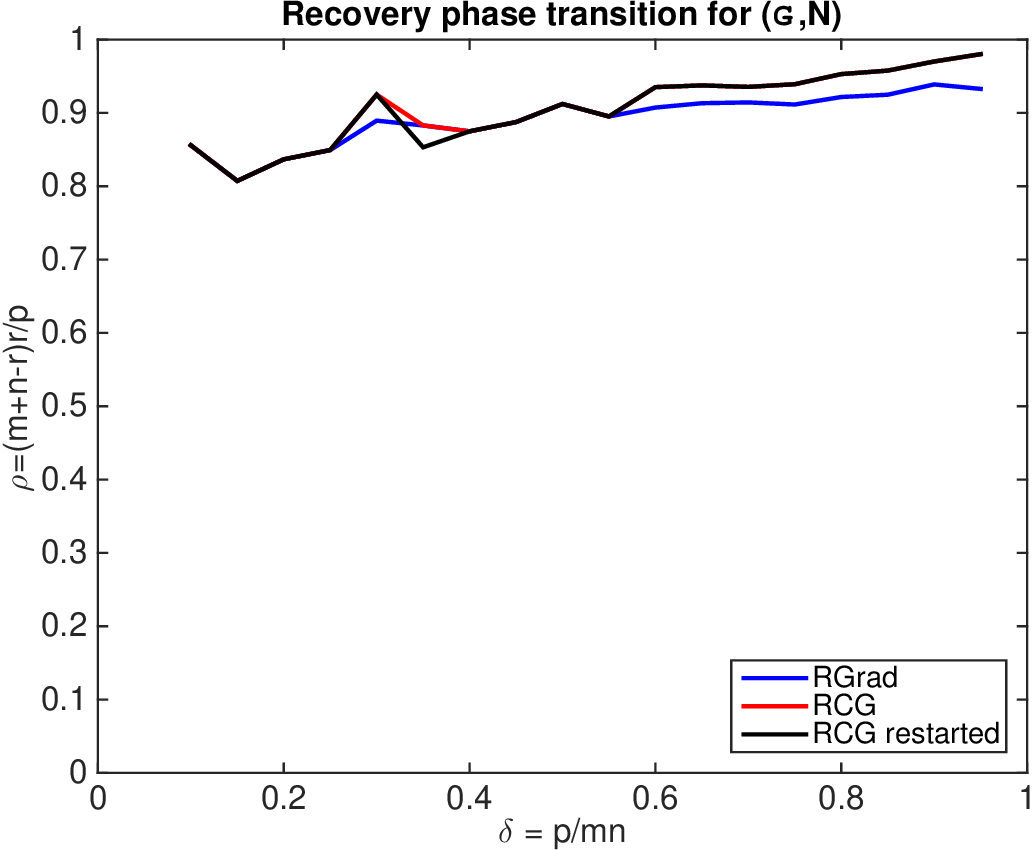}  &
\includegraphics[scale=.35]{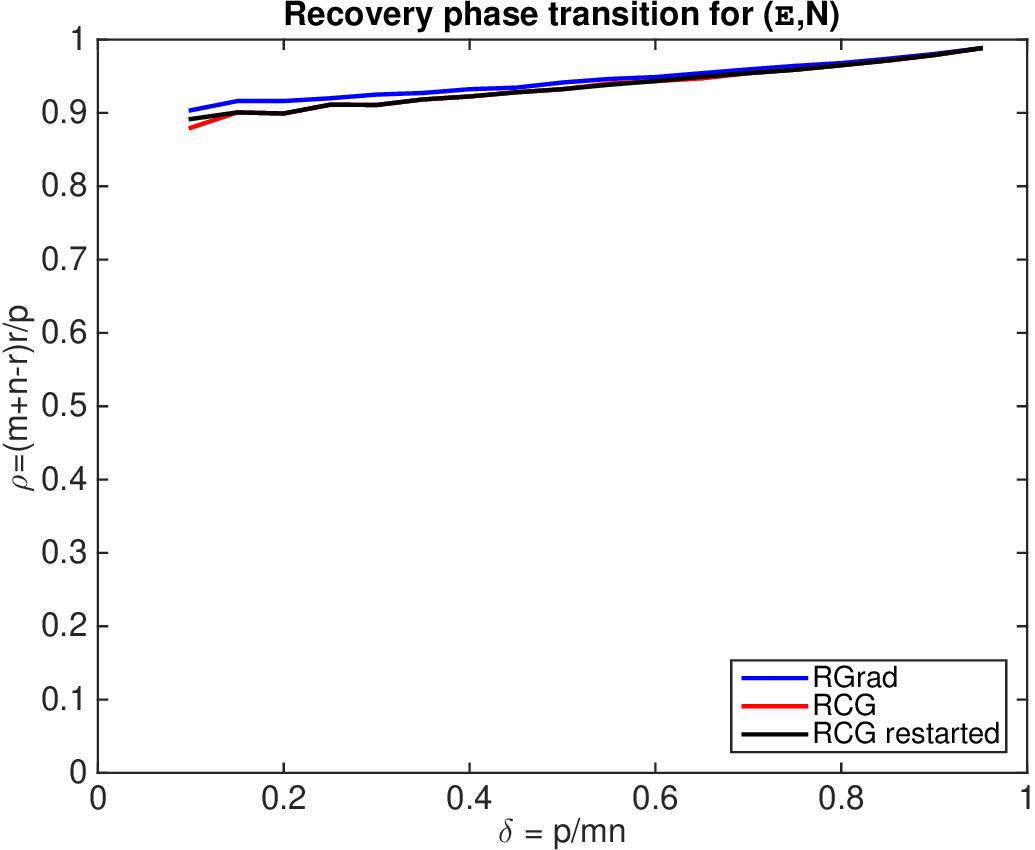} \\
\vspace{0.1in}
(a)&(b) 
\end{tabular}
\vspace{-.2in}
\caption{Empirical phase transition curves for matrix recovery algorithms: RGrad, RCG, RCG started. Horizontal axis $\delta$
and vertical axis $\rho$ as defined in \eqref{eq:delta_rho}. (a) $\G$ with $m = n =80$, (b)
$\E$ with $m=n=800$.}\label{phase_mc}
\end{figure}

\subsection{Computation Time}
\begin{figure}[htp]
\centering
\begin{tabular}{cc}
\vspace{0.1in}
\includegraphics[scale=.35]{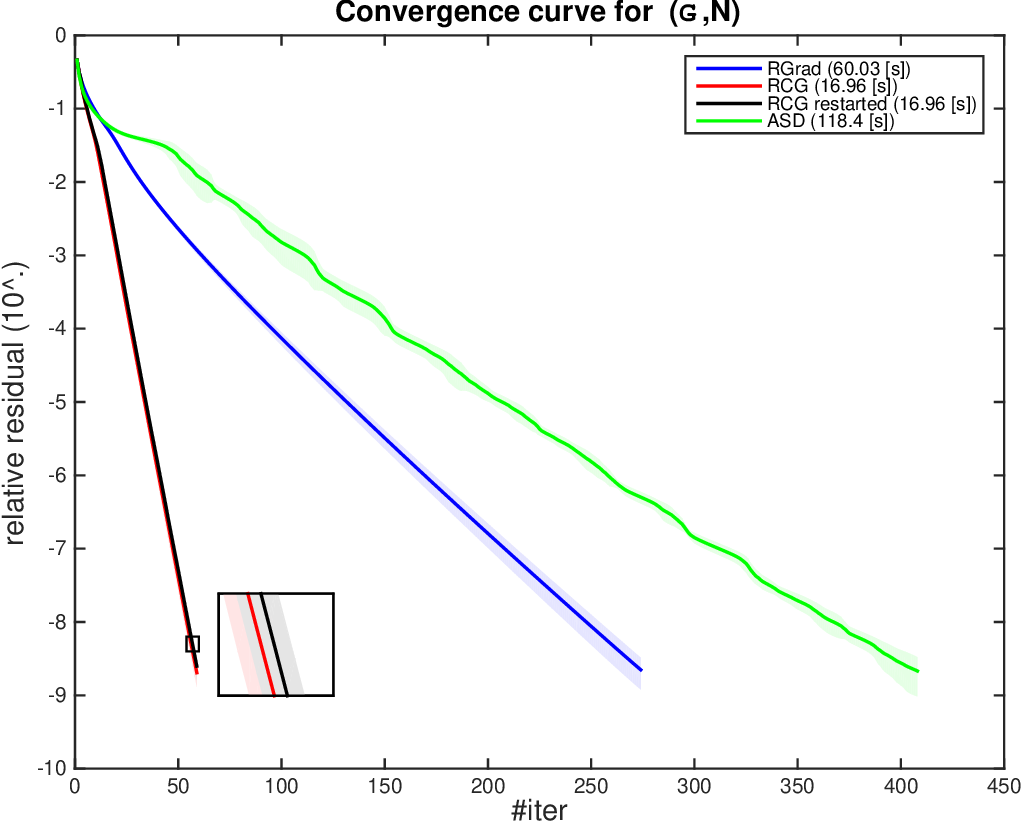}  &\quad
\includegraphics[scale=.35]{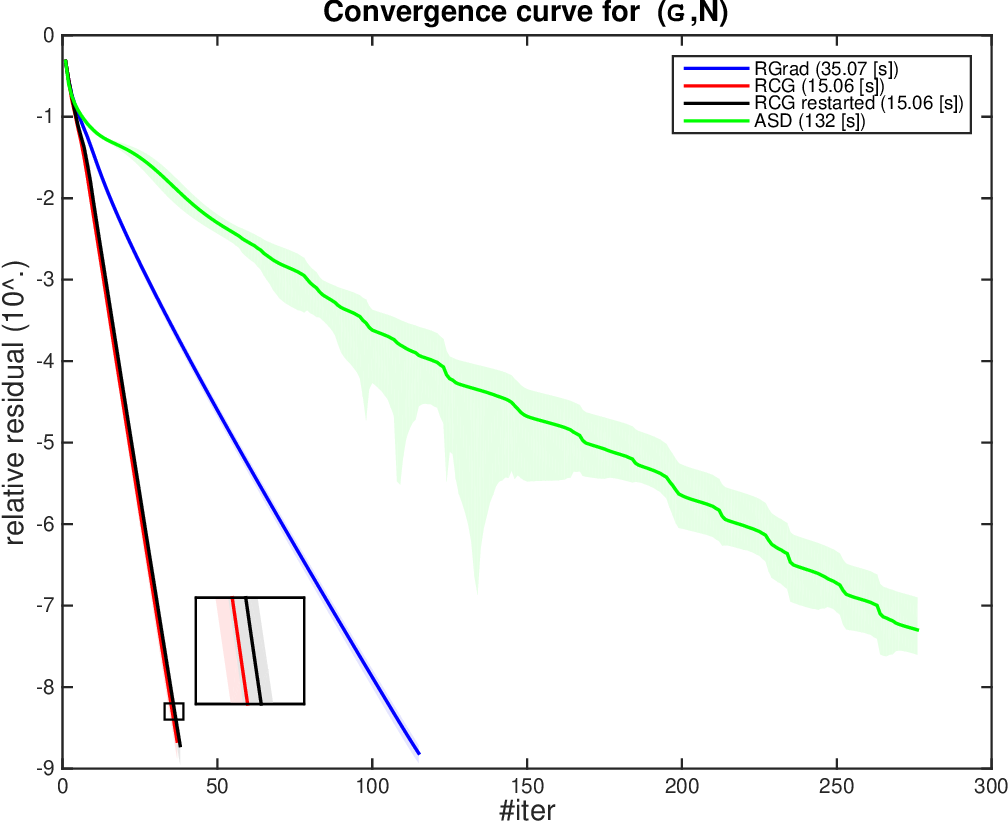} \\
\vspace{0.1in}
(a) & (b)\\
\vspace{0.1in}
\includegraphics[scale=.35]{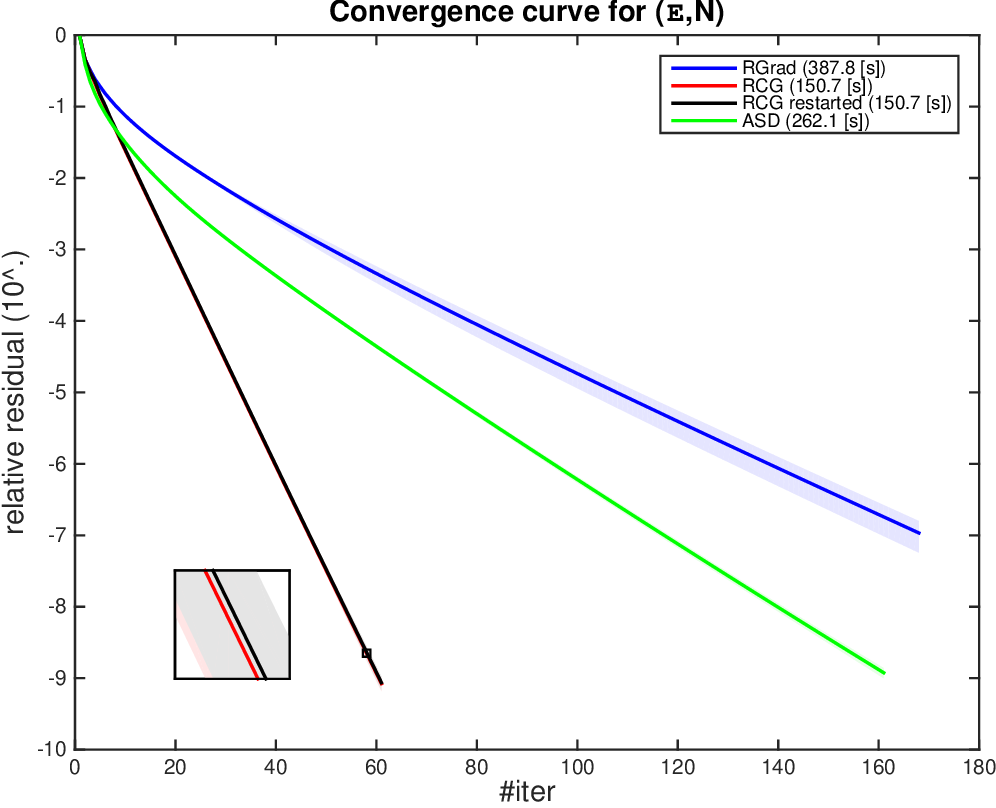}  &\quad
\includegraphics[scale=.35]{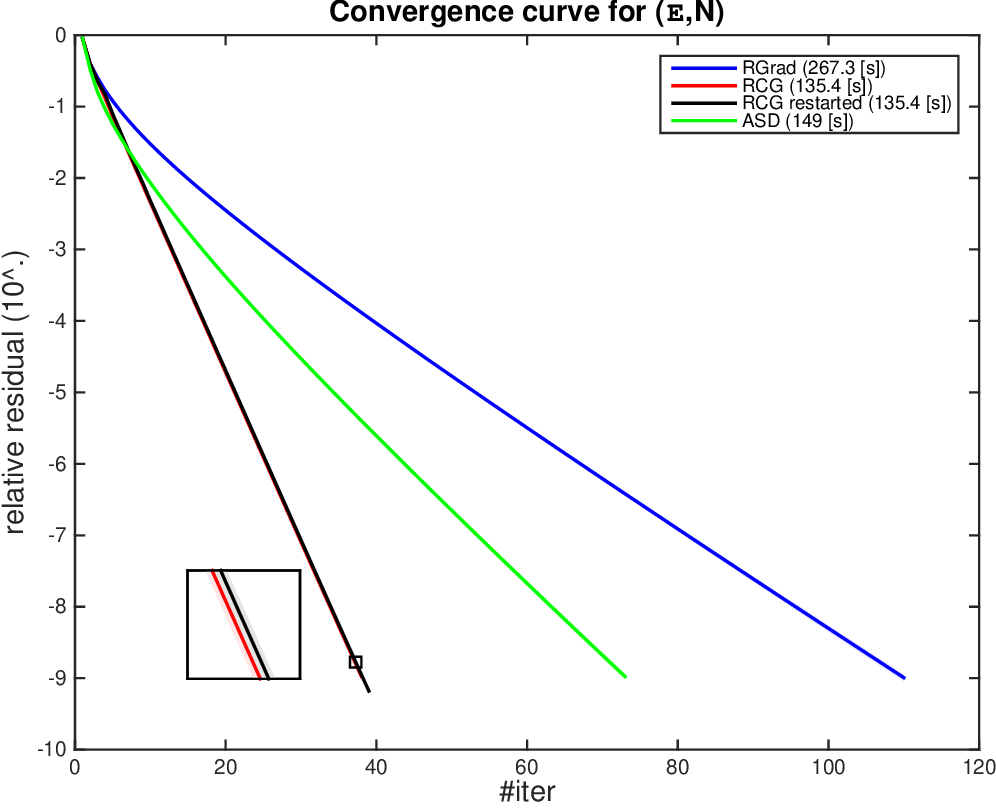} \\
(c) & (d)
\end{tabular}
\caption{Relative residual (mean and standard deviation over ten random tests) as function of number of iterations for Gaussian sensing ((a) and (b)) and entry sensing ((c) and (d)). In Gaussian sensing, $m=n=80$, $r=10$, $1/\rho=2$ (a) or $1/\rho=3$ (b). In entry sensing, $m=n=8000$, $r=100$, $1/\rho=2$ (c) or $1/\rho=3$ (d). The values after each algorithm are the average computational time (seconds) for convergence.}\label{fig:comput}
\end{figure}
Many algorithms have been designed for the matrix recovery problem, for example \cite{CGIHT, jmd2010svp,cevher2012matrixrecipes,tw2012nihtmc,leebresler2011admira,haldahernando2009pf, wyz2012lmafit, tannerwei_asd,bart2012riemannian,ngosaad2012scgrassmc,mishra_geometry,mmbs2012fixrank}, just to name a few. Exhaustive 
comparisons with all those algorithms are impossible. In this section, we will compare RGrad, RCG, RCG restarted with the alternating steepest descend (ASD)
method developed in \cite{tannerwei_asd}.\footnote{We do not compare the Riemannian gradient descent and conjugate gradient descent algorithms (Algs.~\ref{alg:gniht} and \ref{alg:gcgiht}) with NIHT and CGIHT (Algs.~\ref{alg:niht} and \ref{alg:cgiht}) as the superiority of the Riemannian optimization algorithms is very clear following from the discussions in Sec.~\ref{sec:iht}.}  
 ASD takes advantage of the product factorization of low rank matrices and minimizes the bi-quadratic function 
\begin{equation*}f(L,R)=\frac{1}{2}\ln y-\A(LR)\rn_2^2\end{equation*} alternatively 
with $L$ and $R$, where $L\in\R^{m\times r}$ and $R\in\R^{r\times n}$. In each iteration of ASD, it applies  a step of steepest gradient descent on one factor matrix while the other one is held fixed.  The efficiency of ASD has been 
reported in \cite{tannerwei_asd} for  its low per iteration computational complexity. Indirect comparisons with other algorithms can be made from \cite{tannerwei_asd} and references therein.

We compare the algorithms on both Gaussian sensing and entry sensing. For Gaussian sensing, the tests are conducted for $m=n=80$, $r=10$ and 
$1/\rho\in\{2,3\}$; and for entry sensing, the tests are conducted for $m=n=8000$, $r=100$ and $1/\rho\in\{2,3\}$. The algorithms are terminated when the relative residual is less than $10^{-9}$. The relative residual plotted against the number of iterations is presented in Fig.~\ref{fig:comput}. First it can be observed that the convergence curves for RCG and RCG restarted are almost indistinguishable, differing only in one or two iterations, which again shows the effectiveness of the restarting conditions. A close look at the computational results reveals that restarting usually occurs in the first few iterations for RCG restarted.  Moreover, RCG and RCG restarted are sufficiently faster than RGrad and ASD both in terms of the convergence rate and in terms of the average computation time.

\section{Proofs of Main Results}\label{sec:analysis}
\subsection{A Key Lemma}
The following lemma will be used repeatedly, which contains the second order information of 
the smooth low rank matrix manifold, see Fig.~\ref{fig:curve}.

\begin{figure}[htp]
\centering
\includegraphics[scale=0.45, trim = 0 0 0 0, clip]{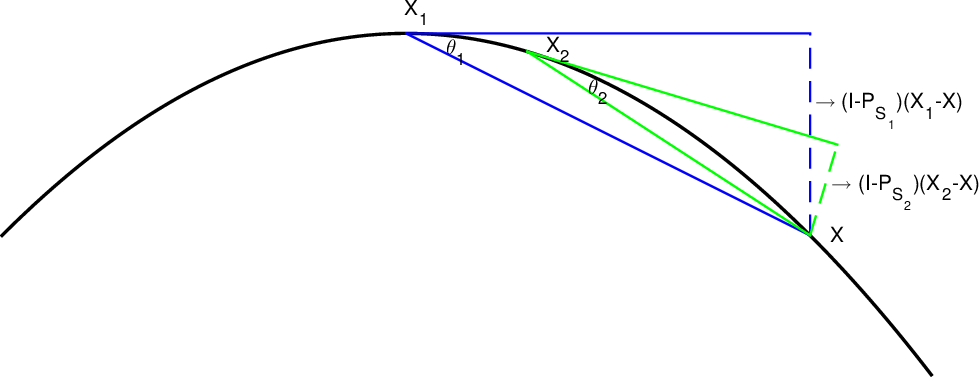}
\caption{The set of low rank matrices  forms a continuous manifold (denoted by a curve here) in the ambient  space. From the plot, it is clear that $\ln\lb I-\P_{\S_1}\rb\lb X_1-X\rb\rn_F=\ln X_1-X\rn_F\cdot\sin\theta_1$ and 
$\ln\lb I-\P_{\S_2}\rb\lb X_2-X\rb\rn_F=\ln X_2-X\rn_F\cdot\sin\theta_2$. As $X_l$ approaches $X$ ($l=1,2$), the angle also decreases. So $\ln\lb I-\P_{\S_l}\rb\lb X_l-X\rb\rn_F=o(\ln X_l-X\rn_F)$.}\label{fig:curve}
\end{figure}

\begin{lemma}\label{lem:proj_error}
Let $X_l=U_l\Sigma_l V_l$ be a rank $r$ matrix, and $\S_l$ be the tangent space of the rank $r$ matrix manifold 
at $X_l$. Let $X$ be another rank $r$ matrix. Then 
{\begin{align*}
\ln \lb I-\P_{\S_l}\rb X\rn_F&\leq \frac{1}{\sigmin{X}}\ln X_l-X\rn_2\ln X_l-X\rn_F\numberthis\label{eq:proj_error_2}\\
&\leq  \frac{1}{\sigmin{X}}\ln X_l-X\rn_F^2.\numberthis\label{eq:proj_error_F}
\end{align*}}
\end{lemma}
The proof of Lem.~\ref{lem:proj_error} relies on the following result which bounds the projection distance of the singular vector subspaces of two matrices. 
\begin{lemma}\label{lem:for_proj_error}
Let $X_l=U_l\Sigma_l V_l^*$ and $X=U\Sigma V^*$ be two rank $r$ matrices. Then 
\begin{equation}\label{eq:for_proj_error_01}
\ln U_lU_l^*-UU^*\rn_2\leq\frac{\ln X_l-X\rn_2}{\sigmin{X}}\quad\mbox{and}\quad\ln V_lV_l^*-VV^*\rn_2\leq\frac{\ln X_l-X\rn_2}{\sigmin{X}};
\end{equation}
\begin{equation}\label{eq:for_proj_error_02}
\ln U_lU_l^*-UU^*\rn_F\leq\frac{\sqrt{2}\ln X_l-X\rn_F}{\sigmin{X}}\quad\mbox{and}\quad\ln V_lV_l^*-VV^*\rn_F\leq\frac{\sqrt{2}\ln X_l-X\rn_F}{\sigmin{X}}.
\end{equation}
\end{lemma}
The proof of Lem.~\ref{lem:for_proj_error} is presented in Appendix~\ref{app:for_proj_error}.
\begin{proof}[Proof of Lemma~\ref{lem:proj_error}] Let $\S$ be the tangent space of the rank $r$ matrix manifold at $X$ as defined in \eqref{eq:tangent_space}. Clearly we have 
$\P_{\S}\lb X\rb=X$. So 
\begin{align*}
\lb I-\P_{\S_l}\rb X & = \lb \P_{\S}-\P_{\S_l}\rb X \\
&=UU^*X+XVV^*-UU^*XVV^*-U_lU_l^*X-XV_lV_l^*+U_lU_l^*XV_lV_l^*\\
&=\lb UU^*-U_lU_l^*\rb X+X\lb VV^*-V_lV_l^*\rb\\
&\quad-UU^*XVV^*+UU^*XV_lV_l^*-UU^*XV_lV_l^*+U_lU_l^*XV_lV_l^*\\
&=\lb UU^*-U_lU_l^*\rb X+X\lb VV^*-V_lV_l^*\rb\\
&\quad-\lb UU^*-U_lU_l^*\rb XV_lV_l^*-UU^*X\lb VV^*-V_lV_l^*\rb\\
&=\lb UU^*-U_lU_l^*\rb X\lb I-V_lV_l^*\rb+\lb I-UU^*\rb X\lb VV^*-V_lV_l^*\rb\\
&=\lb UU^*-U_lU_l^*\rb X\lb I-V_lV_l^*\rb\\
&=\lb UU^*-U_lU_l^*\rb (X-X_l)\lb I-V_lV_l^*\rb,\numberthis\label{eq:in_proof_02}
\end{align*}
where the last two equalities follow from the fact $\lb I-UU^*\rb X=0$ and $X_l(I-V_lV_l^*)=0$.
Taking the Frobenius norm on both sides of \eqref{eq:in_proof_02} gives 
{\begin{align*}
\ln\lb I-\P_{\S_l}\rb X\rn_F &\leq \ln UU^*-U_lU_l^*\rn_2\ln X_l-X\rn_F\ln I-V_lV_l^*\rn_2\\
&\leq \frac{1}{\sigmin{X}}\ln X_l-X\rn_2\ln X_l-X\rn_F,
\end{align*}}
which completes the proof of \eqref{eq:proj_error_2}. Inequality \eqref{eq:proj_error_F} follows directly from the fact $\ln X_l-X\rn_2\leq \ln X_l-X\rn_F$.
\end{proof}
\subsection{Proof of Theorem~\ref{thm:gniht_dense}}
Though Alg.~\ref{alg:gniht} can be viewed as a  restarted variant of Alg.~\ref{alg:gcgiht} in which restarting occurs in each 
iteration, we choose to provide a separated proof to Thm.~\ref{thm:gniht_dense} in order to highlight the main architecture  of the proof.  
We first list two useful lemmas.
\begin{lemma}
Let $Z_1,~Z_2$ be two low rank matrices. Suppose $\la Z_1,Z_2\ra=0$ and $\rank(Z_1)+\rank(Z_2)\leq\min\lb m,n\rb$. Then 
\begin{equation*}\label{lem:for_gniht_dense_01}
\lab \la\A\lb Z_1\rb,\A\lb Z_2\rb\ra\rab\leq R_{\rank(Z_1)+\rank(Z_2)}\ln Z_1\rn_F\ln Z_2\rn_F.
\end{equation*}
\end{lemma}
This lemma is analogous to Lem.~$2.1$ in \cite{candes2008ric} for compressed sensing. The proof for the low rank matrix version can be found in \cite{candesplan2009oracle}. We repeat the proof in Appendix \ref{app:for_gniht_dense} to keep the paper self-contained.
\begin{lemma}\label{lem:for_gniht_dense_02}
Let $X_l=U_l\Sigma_l V_l^*$ be a rank $r$ matrix with the tangent space $\S_l$. Let $X$ be another  rank $r$ matrix. Then the Frobenius norm of $\P_{\S_l}\A^*\A\lb I-\P_{\S_l}\rb X$ 
can be bounded as 
\begin{equation*}
\ln \P_{\S_l}\A^*\A\lb I-\P_{\S_l}\rb(X)\rn_F\leq R_{3r}\ln \lb I-\P_{\S_l}\rb (X)\rn_F.
\end{equation*}
\end{lemma}
The proof of Lem.~\ref{lem:for_gniht_dense_02} is also presented in Appendix \ref{app:for_gniht_dense}.
\begin{proof}[Proof of Theorem~\ref{thm:gniht_dense}]
The proof begins with the following inequality 
\begin{equation*}
\ln X_{l+1}-X\rn_F\leq \ln X_{l+1}-W_l\rn_F +\ln W_l-X\rn_F\leq 2\ln W_l-X\rn_F,
\end{equation*}
where the last inequality follows from the fact that $X_{l+1}$ is the best rank $r$ approximation of $W_l$ in the Frobenius norm. Substituting $W_l=X_l+\alpha_l\P_{\S_l} \lb G_l\rb$ 
into the above inequality gives
\begin{align*}
\ln X_{l+1}-X\rn_F & \leq 2\ln X_l+\alpha_l\P_{\S_l} \lb G_l\rb-X\rn_F\\
&=2\ln X_l+\alpha_l\P_{\S_l}\At\lb y-\A\lb X_l\rb\rb-X\rn_F\\
&=2\ln X_l-X-\alpha_l\P_{\S_l}\At\A\lb X_l-X\rb\rn_F\\
&\leq  2\ln X_l-X-\alpha_l\P_{\S_l}\At\A\P_{\S_l}\lb X_l-X\rb\rn_F\\
&\quad+2\alpha_l\ln \P_{\S_l}\At\A\lb I-\P_{\S_l}\rb\lb X_l-X\rb\rn_F\\
&\leq 2\ln \P_{\S_l}\lb X_l-X\rb-\alpha_l\P_{\S_l}\At\A\P_{\S_l}\lb X_l-X\rb\rn_F\\
&\quad +2\ln \lb I-\P_{\S_l}\rb\lb X_l-X\rb\rn_F\\
&\quad +2\alpha_l\ln \P_{\S_l}\At\A\lb I-\P_{\S_l}\rb\lb X_l-X\rb\rn_F\\
&=2\ln \lb\P_{\S_l}-\alpha_l\P_{\S_l}\At\A\P_{\S_l}\rb\lb X_l-X\rb\rn_F+2\ln \lb I-\P_{\S_l}\rb\lb X\rb\rn_F\\
&\quad +2\lab\alpha_l\rab\ln \P_{\S_l}\At\A\lb I-\P_{\S_l}\rb\lb X\rb\rn_F\\
&:=I_1+I_2+I_3,\numberthis\label{eq:gniht_dense_main}
\end{align*}
where the last inequality follows from the fact  $\lb I-\P_{\S_l}\rb\lb X_l\rb=0$. In the following, we will bound $I_1$, $I_2$ and $I_3$ one by one.

{\em Bound of $I_1$.} We first consider the spectral norm of $\P_{\S_l}-\P_{\S_l}\At\A\P_{\S_l}$. Since it is a symmetric operator, we have 
\begin{align*}
\ln \P_{\S_l}-\P_{\S_l}\At\A\P_{\S_l}\rn &=\sup_{\ln Z\rn_F=1}\lab \la \lb\P_{\S_l}-\P_{\S_l}\At\A\P_{\S_l}\rb\lb Z\rb,Z\ra\rab\\
&=\sup_{\ln Z\rn_F=1}\lab \ln \P_{\S_l}\lb Z\rb\rn_F^2-\ln \A\P_{\S_l}\lb Z\rb\rn_2^2\rab\\
&\leq\sup_{\ln Z\rn_F=1}R_{2r}\ln \P_{\S_l}\lb Z\rb\rn_F^2\leq R_{2r},\numberthis\label{eq:spectral_bound_gniht_dense}
\end{align*}
where the first inequality follows the RIC bound of the sensing operator by noting that $\rank\lb \P_{\S_l}\lb Z\rb\rb~\leq~2r$.
The RIC based bound for the descent stepsize $\alpha_l$ can be obtained as 
\begin{equation}\label{eq:stepsize_gniht_dense}
\frac{1}{1+R_{2r}}\leq\alpha_l=\frac{\ln\P_{U_l}\lb G_l\rb\rn_F^2}{\ln \A\P_{U_l}\lb G_l\rb\rn_2^2}\leq\frac{1}{1-R_{2r}}.
\end{equation}
Immediately we have 
\begin{equation}\label{eq:stepsize_gniht_dense1}
\left|\alpha_l-1\right|\leq\frac{R_{2r}}{1-R_{2r}}
\end{equation}
Combining \eqref{eq:spectral_bound_gniht_dense} and \eqref{eq:stepsize_gniht_dense1} gives the bound of the spectral norm of $\P_{\S_l}-\alpha_l\P_{\S_l}\At\A\P_{\S_l}$
\begin{align*}
\ln \P_{\S_l}-\alpha_l\P_{\S_l}\At\A\P_{\S_l}\rn & \leq \ln \P_{\S_l}-\P_{\S_l}\At\A\P_{\S_l}\rn+| 1-\alpha_l|\ln \P_{\S_l}\At\A\P_{\S_l}\rn\\
&\leq R_{2r}+\frac{R_{2r}}{1-R_{2r}}\lb 1+R_{2r}\rb = \frac{2R_{2r}}{1-R_{2r}}\label{eq:operator_bound_gniht_dense}.
\end{align*}
Thus $I_1$ can be bounded as 
\begin{equation}\label{eq:gniht_dense_I1}
I_1\leq\frac{4R_{2r}}{1-R_{2r}}\ln X_l-X\rn_F.
\end{equation}

{\em Bound of $I_2$.} The second term $I_2$ can be bounded as 
{\begin{equation}\label{eq:gniht_dense_I2}
I_2\leq \frac{2}{\sigmin{X}}\ln X_l-X\rn_F^2.
\end{equation}}
by Lem.~\ref{lem:proj_error}.

{\em Bound of $I_3$. } The third term $I_3$ can be bounded by applying Lem~\ref{lem:for_gniht_dense_02} as follows
\begin{align*}
I_3&\leq \frac{2R_{3r}}{1-R_{2r}}\ln \lb I-\P_{\S_l}\rb\lb X\rb\rn_F
\leq \frac{2R_{3r}}{1-R_{2r}}\ln X_l-X\rn_F,\numberthis\label{eq:gniht_dense_I3}
\end{align*}
where the second inequality follows from the fact $\lb I-\P_{\S_l}\rb\lb X_l\rb=0$.

Inserting \eqref{eq:gniht_dense_I1}, \eqref{eq:gniht_dense_I2} and \eqref{eq:gniht_dense_I3} into  \eqref{eq:gniht_dense_main} gives
{\begin{equation}\label{eq:recursive_gniht_dense_null}
\ln X_{l+1}-X\rn_F \leq \lb \frac{4R_{2r}+2R_{3r}}{1-R_{2r}}+\frac{2}{\sigmin{X}}\ln X_l-X\rn_F\rb\ln X_l-X\rn_F.
\end{equation}}

%
{\em Initialization.} Let $X_0=\H_r(\At\lb y\rb)$ and $U_0\in\R^{m\times r}$ be its left singular vectors.  {Define $Q_0\in\R^{m\times 2r}$ as an orthogonal matrix which 
spans the column subspaces of $X_0$ and $X$.  Let $Q_0^\perp$ be the complement of $Q_0$. Since 
\begin{align*}
\ln X_0-\At\lb y\rb\rn_F^2=\ln X_0-\P_{Q_0}\lb\At\lb y\rb\rb\rn_F^2+\ln \P_{Q_0^\perp}\lb \At\lb y\rb\rb\rn_F^2
\end{align*}
and
\begin{align*}
\ln X-\At\lb y\rb\rn_F^2&=\ln X-\P_{Q_0}\lb\At\lb y\rb\rb\rn_F^2+\ln \P_{Q_0^\perp}\lb \At\lb y\rb\rb\rn_F^2,
\end{align*}
the inequality $\ln X_0-\At\lb y\rb\rn_F\leq \ln X-\At\lb y\rb\rn_F$ implies 
\begin{align*}
\ln X_0-\P_{Q_0}\lb\At\lb y\rb\rb\rn_F\leq \ln X-\P_{Q_0}\lb\At\lb y\rb\rb\rn_F.\numberthis\label{eq:approximation_on_Q}
\end{align*}}
So we have 
\begin{align*}
\ln X_0-X\rn_F &\leq \ln X_0-\P_{Q_0}\lb\At\lb y\rb\rb\rn_F+ \ln \P_{Q_0}\lb\At\lb y\rb\rb-X\rn_F\\
&\leq 2\ln \P_{Q_0}\lb\At\lb y\rb\rb-X\rn_F\\
&= 2\ln \lb\P_{Q_0}-\P_{Q_0}\At\A\P_{Q_0}\rb\lb X\rb\rn_F\\
&\leq 2R_{2r} \ln X\rn_F,\numberthis\label{eq:gniht_dense_init_bound}
\end{align*}
where the last 
inequality follows from the RIC based bound of $\P_{Q_0}-\P_{Q_0}\At\A\P_{Q_0}$ which can be similarly obtained as in \eqref{eq:spectral_bound_gniht_dense}.

Define 
{\begin{equation}\label{eq:gniht_dense_gamma}
\gamma=\frac{4R_{2r}+2R_{3r}}{1-R_{2r}}+\frac{4R_{2r}}{\sigmin{X}}\ln X\rn_F.
\end{equation}}
If $\gamma<1$, 
inserting \eqref{eq:gniht_dense_init_bound} into \eqref{eq:recursive_gniht_dense_null} and proof by induction gives 
\begin{equation}\label{eq:recursive_gniht_dense}
\ln X_{l+1}-X\rn_F \leq \gamma\ln X_l-X\rn_F.
\end{equation}
 {Moreover, if 
\begin{equation*}
R_{3r}\leq \frac{\sigmin{X}}{\sigmax{X}}\cdot\frac{1}{12\sqrt{r}},
\end{equation*}
we have 
\begin{align*}
\gamma\leq \frac{6R_{3r}}{1-R_{3r}}+\frac{4R_{3r}\sqrt{r}\sigmax{X}}{\sigmin{X}}<1,\numberthis\label{eq:gniht_ric_cond}
\end{align*}
where we have utilized the fact $R_{2r}\leq R_{3r}$ following from Def.~\ref{def:RIC} and the inequality $\ln X\rn_F\leq\sqrt{r}\sigmax{X}$.}
\end{proof}
\subsection{Proof of Theorem~\ref{thm:gcgiht_dense}}
The following technical lemma which can be found for example in \cite{CGIHT} establishes the convergence of a three term recurrence relation. 
\begin{lemma}\label{lem:gcgiht_recursive}
Suppose $c_0$, $\tau_1,~\tau_2\geq 0$ and let $\mu=\frac{1}{2}\lb \tau_1+\sqrt{\tau_1^2+4\tau_2}\rb$. Assume $0\leq c_1\leq\mu c_0$
and define $0\leq c_l\leq\tau_{1,l-1}c_{l-1}+\tau_{2,l-2}c_{l-2}$ with $\tau_{1,l-1}\leq \tau_1$ and $\tau_{2,l-2}\leq\tau_2$ for $l\geq 2$.  If $\tau_1+\tau_2<1$, then $\mu<1$ and 
\begin{equation}\label{eq:gcgiht_recursive_lemma}
c_l\leq\mu^lc_0.
\end{equation}
\end{lemma}
\begin{lemma}\label{lem:gcgiht_dense_alpha_beta} 
When the inequalities in Eq.~\eqref{eq:restart_cond_rm} are satisfied, we have 
\begin{equation}\label{eq:gcgiht_dense_alpha_beta}
|\beta_l|\leq \frac{\kappa_2R_{2r}}{1-R_{2r}}+\frac{\kappa_1\kappa_2}{1-R_{2r}}\quad\mbox{and}\quad|\alpha_l-1|\leq\frac{R_{2r}}{(1-R_{2r})-\kappa_1(1+R_{2r})}.
\end{equation} 
\end{lemma}
Notice that when restarting occurs  we have $\beta_l=0$ and $\alpha_l$ can be bounded as in \eqref{eq:stepsize_gniht_dense1}. So the bounds in \eqref{eq:gcgiht_dense_alpha_beta} still apply since $\kappa_1\geq 0$ and $\kappa_2\geq 0$.
\begin{proof}[Proof of Lemma~\ref{lem:gcgiht_dense_alpha_beta}]
We first bound $\beta_l$ as follows
\begin{align*}
\lab\beta_l\rab &=\lab\frac{\la\A\P_{\S_l}\lb G_l\rb,\A\P_{\S_l}\lb P_{l-1}\rb\ra}{\ln\A\P_{\S_l}\lb P_{l-1}\rb\rn_2^2}\rab\\
&=\lab\frac{\la\P_{\S_l}\lb G_l\rb,\P_{\S_l}\At\A\P_{\S_l}\lb P_{l-1}\rb\ra}{\ln\A\P_{\S_l}\lb P_{l-1}\rb\rn_2^2}\rab\\
&\leq\lab\frac{\la\P_{\S_l}\lb G_l\rb,\lb \P_{\S_l}-\P_{\S_l}\At\A\P_{\S_l}\rb\lb \P_{\S_l}P_{l-1}\rb\ra}{\ln\A\P_{\S_l}\lb P_{l-1}\rb\rn_2^2}\rab+\lab
\frac{\la\P_{\S_l}\lb G_l\rb,\P_{\S_l}\lb P_{l-1}\rb \ra}{\ln\A\P_{\S_l}\lb P_{l-1}\rb\rn_2^2}\rab\\
&\leq\frac{R_{2r}}{1-R_{2r}}\frac{\ln \P_{\S_l}\lb G_l\rb\rn_F}{\ln \P_{\S_l}\lb P_{l-1}\rb\rn_F}+\frac{1}{1-R_{2r}}\lab \frac{\la\P_{\S_l}\lb G_l\rb,\P_{\S_l}\lb P_{l-1}\rb \ra}{\ln\P_{\S_l}\lb P_{l-1}\rb\rn_F^2}\rab\\
&=\frac{R_{2r}}{1-R_{2r}}\frac{\ln \P_{\S_l}\lb G_l\rb\rn_F}{\ln \P_{\S_l}\lb P_{l-1}\rb\rn_F}+\frac{1}{1-R_{2r}}\lab \frac{\la\P_{\S_l}\lb G_l\rb,\P_{\S_l}\lb P_{l-1}\rb \ra}{\ln\P_{\S_l}\lb G_l\rb\rn_F\ln\P_{\S_l}\lb P_{l-1}\rb\rn_F}\rab \frac{\ln \P_{\S_l}\lb G_l\rb\rn_F}{\ln\P_{\S_l}\lb P_{l-1}\rb\rn_F}\\
&\leq \frac{\kappa_2R_{2r}}{1-R_{2r}}+\frac{\kappa_1\kappa_2}{1-R_{2r}},
\end{align*}
where the second inequality follows from the RIC bounds of $\P_{\S_l}-\P_{\S_l}\At\A\P_{\S_l}$ (see \eqref{eq:spectral_bound_gniht_dense}) and $\A\P_{\S_l}$, and the last
inequality follows from \eqref{eq:restart_cond_rm}. To bound $\alpha_l$, first note 
\begin{align*}
\lab \beta_l\la \P_{\S_l}\lb G_l\rb,\P_{\S_l}\lb P_{l-1}\rb\ra\rab &=\lab \frac{\la\A\P_{\S_l}\lb G_l\rb,\A\P_{\S_l}\lb P_{l-1}\rb\ra}{\ln\A\P_{\S_l}\lb P_{l-1}\rb\rn_2^2}\la \P_{\S_l}\lb G_l\rb,\P_{\S_l}\lb P_{l-1}\rb\ra\rab\\
&\leq\frac{1+R_{2r}}{1-R_{2r}}\frac{\ln\P_{\S_l}\lb G_l\rb\rn_F}{\ln\P_{\S_l}\lb P_{l-1}\rb\rn_F}\lab \la \P_{\S_l}\lb G_l\rb,\P_{\S_l}\lb P_{l-1}\rb\ra\rab\\
&\leq \frac{\kappa_1\lb1+R_{2r}\rb}{1-R_{2r}}\ln \P_{\S_l}\lb G_l\rb\rn_F^2,
\end{align*} 
where the last inequality follows from  \eqref{eq:restart_cond_rm}. Consequently,
\begin{align*}
\lab\la \P_{\S_l}\lb P_l\rb,\P_{\S_l}\lb G_l\rb\ra\rab &\geq \ln \P_{\S_l}\lb G_l\rb\rn_F^2 - \lab \beta_l\la \P_{\S_l}\lb G_l\rb,\P_{\S_l}\lb P_{l-1}\rb\ra\rab\\
&\geq \lb 1- \frac{\kappa_1\lb1+R_{2r}\rb}{1-R_{2r}}\rb\ln \P_{\S_l}\lb G_l\rb\rn_F^2
\end{align*}
and the application of the Cauchy-Schwarz inequality gives 
\begin{equation*}
\ln \P_{\S_l}\lb G_l\rb\rn_F\leq\frac{1}{ 1- \frac{\kappa_1\lb1+R_{2r}\rb}{1-R_{2r}}}\ln \P_{\S_l}\lb P_l\rb\rn_F.
\end{equation*}
Since $\alpha_l$ can be rewritten as  
\begin{align*}
\alpha_l &=\frac{\la \P_{\S_l}\lb G_l\rb,\P_{\S_l}\lb P_l\rb\ra}{\ln\A\P_{\S_l}\lb P_l\rb\rn_2^2}\\
&=1+\frac{\la \P_{\S_l}\lb G_l\rb,\P_{\S_l}\lb P_l\rb\ra- \la \A\P_{\S_l}\lb G_l\rb,\A\P_{\S_l}\lb P_l\rb\ra}{\ln\A\P_{\S_l}\lb P_l\rb\rn_2^2}\\
&=1+\frac{\la \P_{\S_l}\lb G_l\rb,\lb \P_{\S_l}-\P_{\S_l}\At\A\P_{\S_l}\rb\P_{\S_l}\lb P_l\rb\ra}{\ln\A\P_{\S_l}\lb P_l\rb\rn_2^2},
\end{align*}
it follows that 
\begin{equation*}
\lab\alpha_l-1\rab\leq\frac{R_{2r}}{1-R_{2r}}\frac{\ln \P_{\S_l}\lb G_l\rb\rn_F}{\ln \P_{\S_l}\lb P_l\rb\rn_F}\leq\frac{R_{2r}}{\lb1-R_{2r}\rb-\kappa_1\lb 1+R_{2r}\rb},
\end{equation*}
which completes the proof. 
\end{proof}

\begin{proof}[Proof of Theorem~\ref{thm:gcgiht_dense}]
Analogous to  \eqref{eq:gniht_dense_main}, we have
\begin{align*}
\ln X_{l+1}-X\rn_F & \leq 2\ln X_l+\alpha_lP_l-X\rn_F\\
&=2\ln X_l-X-\alpha \P_{\S_l}\At\A\lb X_l-X \rb+\alpha_l\beta_l\P_{\S_l}\lb P_{l-1}\rb\rn_F\\
&\leq 2\ln \lb\P_{\S_l}-\alpha_l\P_{\S_l}\At\A\P_{\S_l}\rb\lb X_l-X\rb\rn_F +2\ln \lb I-\P_{\S_l}\rb\lb X\rb\rn_F\\
&\quad +2\lab\alpha_l\rab\ln \P_{\S_l}\At\A\lb I-\P_{\S_l}\rb\lb X\rb\rn_F+ 2\lab \alpha_l\rab\lab\beta_l\rab\ln \P_{\S_l}\lb P_{l-1}\rb\rn_F\\
&:= I_4+I_5+I_6+I_7.
\end{align*}
The bound for $I_5$ is exactly the same as the bound for $I_2$, while $I_4$ and $I_6$ can be similarly bounded as $I_1$ and $I_3$, differing only in the bound for $\alpha_l$.
Combining the bounds for $\alpha_l$ \eqref{eq:gcgiht_dense_alpha_beta} and the spetral norm of $\P_{\S_l}-\P_{\S_l}\At\A\P_{\S_l}$ \eqref{eq:spectral_bound_gniht_dense} together gives the bound for the 
spectral norm of $\P_{\S_l}-\alpha_l\P_{\S_l}\At\A\P_{\S_l}$ 
\begin{align*}
\ln \P_{\S_l}-\alpha_l\P_{\S_l}\At\A\P_{\S_l}\rn & \leq \ln \P_{\S_l}-\P_{\S_l}\At\A\P_{\S_l}\rn+\lb 1-\alpha_l\rb\ln \P_{\S_l}\At\A\P_{\S_l}\rn\\
&\leq R_{2r}+\varepsilon_\alpha\lb 1+R_{2r}\rb,\label{eq:operator_bound_gcgiht_dense}
\end{align*}
where 
\begin{equation*}\varepsilon_\alpha := \frac{R_{2r}}{(1-R_{2r})-\kappa_1(1+R_{2r})}.\end{equation*}
So $I_4$ can be bounded as 
\begin{equation}\label{eq:gcgiht_dense_I4}
I_4\leq 2\lb R_{2r}+\varepsilon_\alpha\lb 1+R_{2r}\rb\rb\ln X_l-X\rn_F.
\end{equation}
Inserting the bound for $\alpha_l$ into $I_6$, together with Lem.~\ref{lem:for_gniht_dense_02} gives
\begin{align*}
I_6 &\leq 2R_{3r}\lb 1+\varepsilon_\alpha\rb\ln\lb I-\P_{\S_l}\rb\lb X\rb\rn_F\leq 2R_{3r}\lb 1+\varepsilon_\alpha\rb\ln X_l-X\rn_F.\numberthis\label{eq:gcgiht_dense_I6}
\end{align*}

To bound $I_7$, first note that $\beta_l\P_{\S_l}P_{l-1}$ can be expressed in terms of all the previous gradients 
\begin{equation}\label{eq:gcgiht_dense_P}
\beta_l\P_{\S_l}P_{l-1} = \sum_{j=0}^{l-1}\prod_{q=j+1}^l\beta_q\prod_{k=j}^l \P_{\S_k}\lb G_j\rb,\quad l\geq 1.
\end{equation}
Inserting \eqref{eq:gcgiht_dense_P} into $I_7$ gives
\begin{align*}
I_7 &\leq 2\lab\alpha_l\rab\sum_{j=0}^{l-1}\prod_{q=j+1}^l\lab\beta_q\rab\ln\P_{\S_j}\lb G_j\rb\rn_F\\
&\leq 2\lb 1+\varepsilon_\alpha\rb \sum_{j=0}^{l-1}\varepsilon_\beta^{l-j}\ln \P_{\S_j}\At\A\lb X_j-X\rb\rn_F\\
&\leq 2\lb 1+\varepsilon_\alpha\rb\lb 1+R_{2r}\rb \sum_{j=0}^{l-1}\varepsilon_\beta^{l-j}\ln  X_j-X\rn_F,
\end{align*}
where in the second inequality 
\begin{equation*}\varepsilon_\beta:=\frac{\kappa_2R_{2r}}{1-R_{2r}}+\frac{\kappa_1\kappa_2}{1-R_{2r}},
\end{equation*} and 
the third inequality follows from 
\begin{align*}
\ln \P_{\S_j}\At\A\lb X_j-X\rb\rn_F&=\sup_{\ln Z\rn_F=1}\la \P_{\S_j}\At\A\lb X_j-X\rb,Z\ra\\
&=\sup_{\ln Z\rn_F=1}\la\A\lb X_j-X\rb,\A\P_{\S_j}\lb Z\rb\ra\\
&\leq\sup_{\ln Z\rn_F=1}\ln \A\lb X_j-X\rb\rn_F\ln \A\P_{\S_j}\lb Z\rb\rn\\
&\leq \sup_{\ln Z\rn_F=1}\lb 1+R_{2r}\rb\ln X_j-X\rn_F\ln \P_{\S_j}\lb Z\rb\rn_F\\
&\leq \lb 1+R_{2r}\rb\ln X_j-X\rn_F.
\end{align*}

Combining the bounds for $I_4, ~I_5, ~I_6$ and $I_7$ together gives
\begin{align*}
\ln X_{l+1}-X\rn_F&\leq \lb 2(R_{2r}+R_{3r})(1+\varepsilon_\alpha)+2\varepsilon_\alpha+\frac{2}{\sigmin{X}}\ln X_l-X\rn_F\rb\ln X_l-X\rn_F\\
&\quad+2\lb 1+\varepsilon_\alpha\rb\lb 1+R_{2r}\rb \sum_{j=0}^{l-1}\varepsilon_\beta^{l-j}\ln  X_j-X\rn_F,\quad l\geq 1.
\end{align*}
When $l=0$, Alg.~\ref{alg:gcgiht} is exactly the same as Alg.~\ref{alg:gniht}, so it follows from \eqref{eq:recursive_gniht_dense_null} that
{\begin{equation}\label{eq:gcgiht_dense_recursive_l0}
\ln X_{1}-X\rn_F \leq \lb \frac{4R_{2r}+2R_{3r}}{1-R_{2r}}+\frac{2}{\sigmin{X}}\ln X_0-X\rn_F\rb\ln X_0-X\rn_F.
\end{equation}}

Define  $c_0= \ln X_0-X\rn_F$, 
{\begin{equation}\label{eq:gcgiht_dense_recursive_c1}
c_1 = \lb \frac{4R_{2r}+2R_{3r}}{1-R_{2r}}+\frac{2}{\sigmin{X}}c_0\rb c_0
\end{equation}}
and 
{\begin{align*}
c_{l+1} &= 
 \lb 2(R_{2r}+R_{3r})(1+\varepsilon_\alpha)+2\varepsilon_\alpha+\frac{2}{\sigmin{X}}c_l\rb c_l\\
&\quad+2\lb 1+\varepsilon_\alpha\rb\lb 1+R_{2r}\rb \sum_{j=0}^{l-1}\varepsilon_\beta^{l-j}c_j,\quad l\geq 1,\quad l\geq 1.\numberthis\label{eq:gcgiht_dense_recursive_cl}
\end{align*}}Then it is clear that $c_l\geq\ln X_l-X\rn_F$ for all $l\geq 0$. Morover, Eq.~\eqref{eq:gcgiht_dense_recursive_cl} can be rewritten in a three term  recurrence relation
\begin{align*}
c_{l+1}&=\lb \varepsilon_l+\varepsilon_\beta\rb c_l + \varepsilon_\beta\lb 2\lb 1+\varepsilon_\alpha\rb\lb 1+R_{2r}\rb-\varepsilon_{l-1}\rb c_{l-1}\\
&\leq \lb \varepsilon_l+\varepsilon_\beta\rb c_l + 2\varepsilon_\beta \lb 1+\varepsilon_\alpha\rb\lb 1+R_{2r}\rb c_{l-1},\numberthis\label{eq:gcgiht_dense_three}
\end{align*}
where
{
\begin{align*}
&\varepsilon_0= \frac{4R_{2r}+2R_{3r}}{1-R_{2r}}+\frac{2}{\sigmin{X}}c_0,\\
&\varepsilon_l =  2(R_{2r}+R_{3r})(1+\varepsilon_\alpha)+2\varepsilon_\alpha+\frac{2}{\sigmin{X}}c_l,\quad l\geq 1
\end{align*}}
Define
\begin{align*}
&{\tau_1 = 2(R_{2r}+R_{3r})(1+\varepsilon_\alpha)+2\varepsilon_\alpha+\frac{4R_{2r}}{\sigmin{X}}\ln X\rn_F+\varepsilon_\beta},\\
&\tau_2=2\varepsilon_\beta \lb 1+\varepsilon_\alpha\rb\lb 1+R_{2r}\rb,\\
&\mu=\frac{1}{2}\lb\tau_1+\sqrt{\tau_1^2+4\tau_2}\rb.
\end{align*}
Inequality~\eqref{eq:gniht_dense_init_bound} implies $c_0\leq 2R_{2r}\ln X\rn_F$. So together with the right inequality of \eqref{eq:gcgiht_dense_alpha_beta}, we have $\varepsilon_0\leq\tau_1<\mu$.
Thus if 
\begin{equation}
\gamma:=\tau_1+\tau_2 < 1,
\end{equation}
we have  $\mu<1$, and $c_1<c_0$, which in turn implies 
{\begin{equation*}
\varepsilon_1+\varepsilon_\beta \leq \lb 2(R_{2r}+R_{3r})(1+\varepsilon_\alpha)+2\varepsilon_\alpha+\frac{2}{\sigmin{X}}c_0\rb+\varepsilon_\beta<\tau_1.
\end{equation*}}Therefore the application of Lem.~\ref{lem:gcgiht_recursive} together with proof by induction implies 
 \begin{equation*}
 c_l\leq \mu^l c_0
 \end{equation*}
which completes the proof of the first part of Thm.~\ref{thm:gcgiht_dense}.
 
 When $\kappa_1=0.1$ and $\kappa_2=1$, the sufficient condition for $\gamma<1$ can be verified similarly to \eqref{eq:gniht_ric_cond}.
\end{proof}

\section{Discussion and Future Direction}\label{sec:discuss}
This paper presents theoretical recovery guarantees of a class of Riemannian gradient descent and conjugate gradient algorithms for low rank matrix recovery in terms 
of the restricted isometry constant of the sensing operator. The main results in Thms.~\ref{thm:gniht_dense} and ~\ref{thm:gcgiht_dense} depend on the condition number and the rank of the 
measured matrix. To eliminate the dependence on the condition number, the deflation or stagewise technique in \cite{JaNeSa2012ammc} may be similarly applicable for  Algs.~\ref{alg:gniht} and \ref{alg:gcgiht}.
However, it should be interesting to develop and analyse preconditioned Riemannian gradient  descent and conjugate gradient descent algorithms since they are more favourable in practice. On the 
other hand, to eliminate the dependence on $\sqrt{r}$, it may be necessary to study the convergence   rate of the Riemannian optimization algorithms in terms of the matrix operator norm rather than the Frobenius norm.  However, the contraction of  iterates under the matrix operator norm remains a question.
 In this paper, we have discussed a restarted variant of the Riemannian conjugate gradient descent algorithm with the selection of $\beta_l$ being  developed in \cite{CGIHT}, and guarantee analysis for the other selections of $\beta_l$ in \eqref{eq:beta_all}  as well as for different Riemannian metric \cite{mishra_prec,mishra_geometry} is also an interesting research topic. 


The Riemannian gradient descent and conjugate gradient descent algorithms presented in this paper apply equally to other low rank recovery problems with difference measurement models, such as  phase retrieval  \cite{GeSaPhase,wirtinger,phaselift1,kacz_phase}  and blind deconvolution \cite{blind_conv} where the underlying matrix after lifting is rank one. This line of research will be pursued independently in the future.  
Since the condition number of a rank one matrix is always equal to one, 
 it is worth investigating whether we can obtain similar recovery guarantees for phase retrieval and blind deconvolution, but directly in terms of the sampling complexity and  without explicit dependence on the condition number of the underlying matrix. Finally, it may be possible to generalize the notion in low rank matrix manifold, for example the restricted isometry constant, to the abstract framework of Riemannian manifold and then extend the analysis in this paper to more general Riemannian  gradient descent and conjugate gradient descent algorithms.
\section*{Acknowledgments}
KW has been supported by DTRA-NSF grant No.1322393. SL was supported in part by the Hong Kong RGC grant 16303114.
\bibliography{iht}
\bibliographystyle{plain}
\begin{appendix}\label{sec:appen}
\section{Proof of Lemma~\ref{lem:for_proj_error}}\label{app:for_proj_error}
We only prove the left inequalities in \eqref{eq:for_proj_error_01} and \eqref{eq:for_proj_error_02} and 
the right inequalities can be similarly established. 
The left inequality of \eqref{eq:for_proj_error_01} follows from direct calculations 
\begin{align*}
\ln U_lU_l^*-UU^*\rn_2&=\ln UU^*\lb I-U_lU_l^*\rb\rn_2=\ln \lb I-U_lU_l^*\rb UU^*\rn_2\\
&=\ln \lb I-U_lU_l^*\rb XV\Sigma^{-1}U^*\rn_2\\
&=\ln \lb I-U_lU_l^*\rb \lb X_l-X\rb V\Sigma^{-1}U^*\rn_2\\
&\leq \ln I-U_lU_l^*\rn_2\ln X_l-X\rn_2 \ln V\rn_2\ln\Sigma^{-1}\rn_2\ln U^*\rn_2\\
&=\frac{\ln X_l-X\rn_2}{\sigmin{X}}, 
\end{align*}
where the first equality follows from a standard result in textbook, see for example \cite[Thm.~2.6.1]{GolubLoan2013mc}, and
the fourth equality follows from the fact $\lb I-U_lU_l^*\rb X_l=0$.

To prove the left inequality of \eqref{eq:for_proj_error_02}, we first show that 
\begin{equation}\label{eq:in_proof_01}
\ln \lb I-U_lU_l^*\rb UU^*\rn_F=\ln U_lU_l^*\lb I-UU^*\rb\rn_F.
\end{equation}
Equality \eqref{eq:in_proof_01} can be obtained by noting that
\begin{align*}
\ln \lb I-U_lU_l^*\rb UU^*\rn_F^2 &= \la \lb I-U_lU_l^*\rb UU^*, \lb I-U_lU_l^*\rb UU^*\ra\\
&=\la I-U_lU_l^*,UU^*\ra = r-\la U_lU_l^*,UU^*\ra
\end{align*}
and
\begin{align*}
\ln U_lU_l^*\lb I-UU^*\rb\rn_F^2 &= \la U_lU_l^*\lb I-UU^*\rb, U_lU_l^*\lb I-UU^*\rb\ra\\
&=\la U_lU_l^*,I-UU^*\ra = r-\la U_lU_l^*,UU^*\ra.
\end{align*}
So it follows that
\begin{align*}
\ln U_lU_l^*-UU^*\rn_F & = \sqrt{2}\ln \lb I-U_lU_l^*\rb UU^*\rn_F\\
&=\sqrt{2}\ln \lb I-U_lU_l^*\rb XV\Sigma^{-1}U^*\rn_F\\
&=\sqrt{2}\ln \lb I-U_lU_l^*\rb \lb X_l-X\rb V\Sigma^{-1}U^*\rn_F\\
&\leq\sqrt{2} \ln I-U_lU_l^*\rn_2\ln X_l-X\rn_F \ln V\rn_2\ln\Sigma^{-1}\rn_2\ln U^*\rn_2\\
&=\frac{\sqrt{2}\ln X_l-X\rn_F}{\sigmin{X}}.
\end{align*}
\section{Proofs of Lemmas~\ref{lem:for_gniht_dense_01} and \ref{lem:for_gniht_dense_02}}\label{app:for_gniht_dense}
\begin{proof}[Proof of Lemma~\ref{lem:for_gniht_dense_01}]
The proof follows that for  \cite[Lem.~2.1]{candes2008ric} in compressed sensing. Without loss of generality, assume $\ln Z_1\rn_F=1$ and 
$\ln Z_2\rn_F=1$. Then the application of the RIC bounds gives 
\begin{align*}
&\lb 1-R_{rank(Z_1)+rank(Z_2)}\rb \ln Z_1\pm Z_2\rn_F^2 \leq\ln \A\lb Z_1\pm Z_2\rb\rn_2^2,\\
&\lb 1+R_{rank(Z_1)+rank(Z_2)}\rb \ln Z_1\pm Z_2\rn_F^2 \geq\ln \A\lb Z_1\pm Z_2\rb\rn_2^2.
\end{align*}
Since $\la Z_1,Z_2\ra=0$, we have $\ln Z_1\pm Z_2\rn_F^2=2$. So
\begin{equation*}
2\lb 1-R_{rank(Z_1)+rank(Z_2)}\rb \leq\ln \A\lb Z_1\pm Z_2\rb\rn_2^2\leq 2\lb 1+R_{rank(Z_1)+rank(Z_2)}\rb.
\end{equation*}
Finally the parallelogram identity gives 
\begin{equation*}
\lab \la \A\lb Z_1\rb,\A\lb Z_2\rb\ra \rab=\frac{1}{4}\lab \ln \A\lb Z_1+Z_2\rb\rn_2^2-\ln \A\lb Z_1- Z_2\rb\rn_2^2\rab\leq R_{rank(Z_1)+rank(Z_2)},
\end{equation*}
which completes the proof.
\end{proof}
\begin{proof}[Proof of Lemma~\ref{lem:for_gniht_dense_02}]
\begin{align*}
\ln \P_{\S_l}\A^*\A\lb I-\P_{\S_l}\rb (X)\rn_F & =\sup_{\ln Z\rn_F=1}\lab\la \P_{\S_l}\A^*\A\lb I-\P_{\S_l}\rb (X),Z\ra\rab\\
&=\sup_{\ln Z\rn_F=1}\lab\la \A\lb I-\P_{\S_l}\rb (X), \A\P_{\S_l}(Z)\ra\rab\\
&\leq\sup_{\ln Z\rn_F=1} R_{3r}\ln \lb I-\P_{\S_l}\rb (X)\rn_F\ln \P_{\S_l}(Z)\rn_F\\
&\leq R_{3r}\ln \lb I-\P_{\S_l}\rb (X)\rn_F,
\end{align*}
where the second to last inequality follows from Lem.~\ref{lem:for_gniht_dense_01} together with the fact $\rank\lb \lb I-\P_{\S_l}\rb (X)\rb\leq r$ and 
$\rank\lb  \P_{\S_l}(Z)\rb\leq 2r$.
\end{proof}
\end{appendix}
\end{document}